\documentclass[letterpaper,12pt]{article}
\usepackage{latexsym,amssymb,amsfonts,amsmath,amsthm,nccmath,enumitem,setspace,graphicx,float}
\usepackage[dvipsnames]{xcolor}
\usepackage[hidelinks]{hyperref}
\usepackage[margin=2cm]{geometry}
\usepackage{subcaption}
\usepackage{multicol}
\usepackage{dsfont}
\usepackage{bm}
\usepackage{float}
\usepackage[normalem]{ulem}
\newtheorem{theorem}{Theorem}[section]

\newtheorem{lemma}[theorem]{Lemma}

\newtheorem{notation}[theorem]{Notation}

\theoremstyle{definition}
\newtheorem{definition}[theorem]{Definition}
\newtheorem{problem}[theorem]{Problem}

\usepackage{hyperref}

\newcommand{\com}[1]{{\color{blue} #1}}
\newcommand{\mc}{\mathcal}
\newcommand{\Z}{\mathbb{Z}}
\newcommand{\B}{\mathcal{B}}
\newcommand{\D}{\mathcal{D}}
\newcommand{\C}{\mathcal{C}}
\newcommand{\T}{\mathcal{T}}
\renewcommand{\L}{\mathcal{L}}
\newcommand{\e}{\epsilon}
\renewcommand{\l}{\lambda}
\renewcommand{\S}{\mathcal{S}}
\renewcommand{\P}{\mathcal{P}}
\newcommand{\len}{\textrm{len}}

\usepackage[maxfloats=256]{morefloats}

\def \deg {{\rm deg}}
\def \Circ {{\rm Circ}}
\def \OP {{\rm OP}}
\def \arank {{\rm arank}}
\def \diag {{\rm diag}}
\def \floor#1{\lfloor#1\rfloor}
\def \Nbd {{\rm Nbd}}
\def \sgn {{\rm sgn}}
\def \leq {\leqslant}
\def \geq {\geqslant}
\def \F {\mathcal{F}}
\def \mod{\pmod}

\definecolor{color1}{rgb}{0,0, 0.0}
\definecolor{color2}{rgb}{0.55, 0.55, 0.55}
\usepackage{tikz}
\tikzstyle{vertex}=[circle, draw=black, fill=black, minimum size=2pt, inner sep=2]
\usetikzlibrary{arrows, snakes,arrows.meta,calc,decorations.markings,math,arrows.meta}
\newcommand{\midarrow}{\tikz \draw[-{Stealth[scale=1.1]}] (0,0) -- (0.1,0);}

\makeatletter

\tikzset{
  on each segment/.style={
    decorate,
    decoration={
      show path construction,
      moveto code={},
      lineto code={
        \path [#1]
        (\tikzinputsegmentfirst) -- (\tikzinputsegmentlast);
      },
      curveto code={
        \path [#1] (\tikzinputsegmentfirst)
        .. controls
        (\tikzinputsegmentsupporta) and (\tikzinputsegmentsupportb)
        ..
        (\tikzinputsegmentlast);
      },
      closepath code={
        \path [#1]
        (\tikzinputsegmentfirst) -- (\tikzinputsegmentlast);
      },
    },
  },
  mid arrow/.style={postaction={decorate,decoration={
        markings,
        mark=at position .5 with {\arrow[scale=1.2,#1]{stealth}}
      }}},
}

\usetikzlibrary{decorations.markings}
\tikzset{
    set arrow inside/.code={\pgfqkeys{/tikz/arrow inside}{#1}},
    set arrow inside={end/.initial=>, opt/.initial=},
    /pgf/decoration/Mark/.style={
        mark/.expanded=at position #1 with
        {
            \noexpand\arrow[\pgfkeysvalueof{/tikz/arrow inside/opt}]{\pgfkeysvalueof{/tikz/arrow inside/end}}
        }
    },
    arrow inside/.style 2 args={
        set arrow inside={#1},
        postaction={
            decorate,decoration={
                markings,Mark/.list={#2}
            }
        }
    },
}

\makeatother

\newcommand{\comment}[1]{{\color{RoyalPurple} #1}}
\newcommand{\change}[2]{{\color{ForestGreen} #1}{\color{BrickRed} #2}}

\let\oldproofname=\proofname
\renewcommand{\proofname}{\rm\bf{\oldproofname}}

\title{Completing the solution of the directed Oberwolfach problem with two tables}
\author{D. Horsley\footnote{Email: daniel.horsley@monash.edu. Mailing address: School of Mathematical Sciences, Monash University, 9 Rainforest Walk, Clayton, VIC, 3800, Australia.}, Monash University \\   A. Lacaze-Masmonteil\footnote{Email: alaca054@uottawa.ca. Mailing address: Department of Mathematics and Statistics, University of Ottawa,150 Louis-Pasteur Private, Ottawa, ON, K1N 9A7, Canada.}, University of Ottawa }

\begin{document}
\maketitle \baselineskip 17pt

\begin{center}
{\bf Abstract}
\end{center}

We address the last outstanding case of the directed Oberwolfach problem with two tables of different lengths. Specifically, we show that the complete symmetric directed graph $K^*_n$ admits a decomposition into spanning subdigraphs comprised of two vertex-disjoint directed cycles of length $t_1$ and $t_2$, respectively, where $t_1\in \{4,6\}$, $t_2$ is even, and $t_1+t_2\geqslant 14$. In conjunction with recent results of Kadri and Šajna, this gives a complete solution to the directed Oberwolfach problem with two tables of different lengths.


\section{Introduction}\label{S:intro}

In this paper, we investigate a variation of the famous Oberwolfach Problem (OP). Introduced by Ringel \cite{Intro} in 1967, the $\OP(t_1, t_2, \ldots, t_s)$ poses the following question: given $n=2k+1$ people and $t$ round tables that respectively seat $t_1, t_2, \ldots , t_s$ people, where $t_1+t_2+\cdots+t_s=n$ and $t_i \geqslant 3$, does there exist a set of $k$ seating arrangements such that each person sits beside every other person precisely once? This problem can be formulated as a graph-theoretic problem by considering the question of existence of a 2-factorization of the complete graph $K_n$ such that each 2-factor is comprised of cycles of lengths $t_1, t_2, \ldots, t_s$. In \cite{HuaKot}, the OP was adapted to consider the case where $n$ is even. In that case, the existence of a 2-factorization of $K_n-I$  is considered, where  $K_n-I$ is the complete graph with the edges of a 1-factor removed. Constructive solutions to the OP have been found in each of the following cases: cycles of uniform length \cite{AlsHag, AlsSch, HofSch, HuaKot}, two cycles \cite{2tabodd, traetta}, any combination of cycles of even length \cite{BryDan, Hag}, and $n \leqslant 60$ \cite{18to40, 13ver, 13ver1, 60.}. Constructive solutions to the OP have also been found for several infinite families of cases in \cite{AlsBry} and \cite{BryScha}. Lastly, it has also been shown non-constructively that a solution to the OP exists for all sufficiently large $n$ \cite{solved}.

The directed Oberwolfach problem ($\OP^*(t_1, t_2, \ldots, t_s))$ considers a similar scenario. This time, we let $t_i \geqslant 2$ and we seek $n-1$ seating arrangements with the added property that each guest is to be seated to the right of every other guest exactly once. If all $s$ tables are of the same length $t$, we write $\OP^*(t; s)$. When $n$ is odd, one can easily construct a solution to the $\OP^*(t_1, t_2, \ldots, t_s)$ from a solution to the $\OP(t_1, t_2, \ldots, t_s)$. Therefore it suffices to consider the $\OP^*(t_1, t_2, \ldots, t_s)$ for $n$ even in those cases where a solution to the $\OP(t_1, t_2, \ldots, t_s)$ is known.


Recently, the last open case of the $\OP^*(t; s)$ was settled \cite{LacMas}. Thus we have a constructive proof of Theorem \ref{thm:RCMD} below.

\begin{theorem}[\cite{Abel, AdaBry, BenZha, BerGerSot, BurFranSaj, BurSaj, LacMas, Til}]
\label{thm:RCMD}
Let $s$ and $t$ be positive integers such that $t s$ is even. The  $\OP^*(t; s)$ has a solution if and only if $(s, t) \not\in \{(1,6), (1, 4),$ $(2, 3) \}$.
\end{theorem}

Naturally, the next step is to consider the case with cycles of varying length. The only result on this more general case of the $\OP^*$ when $n$ is even can be found in  \cite{SajKad} in which Kadri and Šajna use a recursive approach to obtain several infinite families of solutions. One of the key results of \cite{SajKad}  is a near-complete constructive solution to the directed Oberwolfach problem with two cycles of varying lengths formulated in Theorem \ref{thm:SajKad2} below.

\begin{theorem}[\cite{SajKad}]
\label{thm:SajKad2}
Let $t_1$ and $t_2$ be integers such that $2 \leqslant t_1<t_2$. Then the $\OP^*(t_1, t_2)$ has a solution if and only if $(t_1, t_2) \neq (3,3)$ with a possible exception in the case where $t_1 \in \{4,6\}$, $t_2$ is even, and $t_1+t_2\geqslant 14$.
\end{theorem}

The recursive approach used to prove Theorem \ref{thm:SajKad2} relies on the existence of a solution to $\OP^*(t; 1)$. However, it is known from Theorem \ref{thm:RCMD} that no such decomposition exists when $t_1 \in \{4,6\}$. Therefore, the methods of \cite{SajKad} cannot be used to construct a solution to the $\OP^*(t_1, t_2)$ when $t_1 \in \{4,6\}$ and $t_2$ is even.

Here we complement the results of Theorem \ref{thm:SajKad2} and complete the solution of the directed Oberwolfach problem with two tables.

\begin{theorem}
\label{thm:twotab}
Let $t_1$ and $t_2$ be positive even integers such that $t_1 \in \{4,6\}$ and $t_1+t_2 \geqslant 14$. Then the $\OP^*(t_1, t_2)$ has a solution.
\end{theorem}

Theorems \ref{thm:SajKad2} and \ref{thm:twotab} jointly imply a complete constructive solution to the $\OP^*(t_1, t_2)$ stated below.

\begin{theorem}
\label{thm:final}
Let $t_1$ and $t_2$ be integers such that $2 \leqslant t_1\leqslant t_2$. Then the $\OP^*(t_1, t_2)$ has a solution if and only if $(t_1, t_2) \neq (3,3)$.
\end{theorem}

This paper is structured as follows.  In Section \ref{S:prem}, we give key definitions. Then, in Section \ref{S:strategy}, we take a reduction step by showing that it suffices to find particular 2-factorizations of a class of sparser digraphs. Next, in Section \ref{S:main}, we describe the ingredients needed to obtain the desired 2-factorizations and prove that these indeed give rise to the appropriate solutions of the directed Oberwolfach problem. We conclude by constructing the desired set of ingredients required to form the directed 2-factorizations we need.

\section{Key definitions} \label{S:prem}

We make the standard assumption that all directed graphs (digraphs for short) are strict. This means that digraphs do not contain loops or parallel arcs. If $G$ is a digraph (graph), we shall denote its vertex set as $V(G)$ and its arc set (edge set) as $A(G)$ ($E(G)$), respectively. For any graph $G$, let $G^*$ denote the digraph with vertex set $V(G)$ and arc set $\{(x,y),(y,x):\{x,y\} \in E(G)\}$. Let $K^*_n$ denote the complete symmetric digraph on $n$ vertices and let $\vec{C}_m$ denote the directed cycle on $m$ vertices. Let $E_m$ denote the undirected graph with $m$ vertices and no edges.

The length of a directed path (dipath for short) or a directed cycle refers to the number of arcs it has. For a dipath $P$, we denote its length as $\len(P)$. Moreover, the \textit{source} of a dipath $P$ is the vertex with in-degree 0 and is denoted $s(P)$, while the \textit{terminal} of $P$ is the vertex with out-degree 0 and is denoted $t(P)$.

Let $G$ be a digraph. A \textit{decomposition} of a  $G$ is a set $\{H_1, H_2, \ldots, H_r\}$ of pairwise arc-disjoint subdigraphs of $G$ such that $A(G)=A(H_1) \cup A(H_2) \cup \cdots \cup A(H_r)$. A spanning subdigraph of $G$ that is a disjoint union of directed cycles of $G$ is a \textit{directed 2-factor} of $G$. A \textit{$(\vec{C}_{t_1}, \vec{C}_{t_2}, \ldots, \vec{C}_{t_s})$-factor} of $G$ is a directed 2-factor that is the disjoint union of $s$ directed cycles of lengths $t_1, t_2, \ldots, t_s$. A \textit{bipartite directed  2-factor} of digraph $G$ is a directed 2-factor comprised of directed cycles of even lengths. If $H$ is a spanning subdigraph of $G$ and $G$ admits a decomposition into subdigraphs isomorphic to $H$, then this decomposition is called an \textit{$H$-factorization}. In particular, a \textit{$(\vec{C}_{t_1}, \vec{C}_{t_2}, \ldots, \vec{C}_{t_s})$-factorization} of $G$ is a decomposition of $G$ into $(\vec{C}_{t_1}, \vec{C}_{t_2}, \ldots, \vec{C}_{t_s})$-factors. A \textit{directed 2-factorization} of $G$ is a decomposition of $G$ into directed 2-factors.  All these terms can be analogously defined for undirected graphs.

We now formulate the $\OP^*(t_1, t_2, \ldots, t_s)$ in graph-theoretic terms.

\begin{problem} [$\OP^*(t_1, t_2, \ldots, t_s)$]
For integers $2\leqslant t_1 \leqslant t_2 \leqslant \cdots \leqslant t_s$ such that $t_1+t_2+\cdots+t_s=n$,  does $K^*_n$ admit a $(\vec{C}_{t_1}, \vec{C}_{t_2}, \ldots, \vec{C}_{t_s})$-factorization?
\end{problem}

To prove Theorem \ref{thm:twotab}, we construct a $(\vec{C}_{t_1}, \vec{C}_{t_2})$-factorization of $K^*_n$ when $n=t_1+t_2$, $t_1 \in \{4,6\}$, $t_2$ is even, and $n\geqslant 14$. 

We conclude this section with a pair of definitions that are used to construct the desired $(\vec{C}_{t_1}, \vec{C}_{t_2})$-factorizations of $K^*_n$.

\begin{definition}
For graphs $G$ and $H$, the \emph{wreath product of $G$ with $H$}, denoted $G \wr H$, is the graph with vertex set $V(G) \times V(H)$ in which $(g_1,h_1)$ and $(g_2,h_2)$ are adjacent if and only if either $g_1g_2 \in E(G)$ or $g_1=g_2$ and $h_1h_2 \in E(H)$.
\end{definition}

\begin{definition}
For a subset $S$ of $\{1,\ldots,\lfloor\frac{n}{2}\rfloor\}$,  the \emph{circulant of order $n$ with connection set $S$}, denoted $\Circ(n,S)$, is the graph with vertex set $\mathbb{Z}_n$ and edge set $\{\{i,i+s\}:i \in \mathbb{Z}_n,s \in S\}$ with addition performed modulo $n$.
\end{definition}


\section{Overall strategy} \label{S:strategy}

This section details the overall strategy we follow to prove Theorem \ref{thm:twotab}. Our primary objective is to demonstrate that, in order to construct the desired  $(\vec{C}_{t_1}, \vec{C}_{t_2})$-factorization of $K^*_n$, it suffices to construct a $(\vec{C}_{t_1}, \vec{C}_{t_2})$-factorization of a sparser digraph that only requires seven or nine $(\vec{C}_{t_1}, \vec{C}_{t_2})$-factors.

Crucial to our approach is the following immediate consequence of a lemma of Häggkvist \cite{Hag}.

\begin{lemma}[\cite{Hag}]\label{L:Hag}
Let $D$ be a bipartite directed $2$-factor of order $2m$ comprised of directed cycles of length at least $4$. The digraph  $(C_m \wr E_2)^*$ admits a $D$-factorization.
\end{lemma}

\begin{proof}
Let $F$ be the 2-factor obtained from $D$ by replacing  each arc $(x,y)$ with an undirected edge $\{x,y\}$. By the first lemma of \cite{Hag}, commonly known as Häggkvist's Lemma, there is an $F$-factorization $\F$ of $C_m \wr E_2$. Thus there is an $F^*$-factorization $\mathcal{F}^*$ of $(C_m \wr E_2)^*$. Clearly, each copy of $F^*$ in $\F^*$ can be decomposed into two copies of $D$ and together, these copies of $D$ form the desired directed 2-factorization of $(C_m \wr E_2)^*$.
\end{proof}

Lemma~\ref{L:Hag} does not apply to bipartite 2-factors containing at least one cycle of length 2. However, for our purpose, we do not need to consider this case.

Let $D$ be a bipartite directed $2$-factor on $2m$ vertices comprised of directed cycles of length at least 4. Our overall strategy for finding $D$-factorizations of $K^*_{2m}$ is first to decompose $K^*_{2m}$ into copies of $(C_m \wr E_2)^*$ and a single copy of another graph that we call $W^*_{2m}$. It will then suffice to find a $D$-factorization of $W^*_{2m}$ because we can form a $D$-factorization of $K^*_{2m}$ by taking the union of this $D$-factorization with $D$-factorizations of the copies of $(C_m \wr E_2)^*$ provided by Lemma~\ref{L:Hag}. In the remainder of this section, we define the graph $W^*_{2m}$ and show that $K^*_{2m}$ can indeed be decomposed into copies of $(C_m \wr E_2)^*$ and one copy of $W^*_{2m}$. This approach is inspired by the one used by Bryant and Danziger in \cite{BryDan}.

\begin{definition}\label{defn:W}
If $m$ is odd, we define $W_{2m}$ to be $\Circ(m,\{1,2\}) \wr K_2$ and if $m$ is even, we define $W_{2m}$ to be $\Circ(m,\{1,3^{\rm e}\}) \wr K_2$, where $\Circ(m,\{1,3^{\rm e}\})$ denotes the graph with vertex set $\mathbb{Z}_m$ and edge set
\[\{\{i,i+1\}:i \in \mathbb{Z}_m\} \cup \{\{i,i+3\}:i \in \mathbb{Z}_m \text{ is even}\}.\]
\end{definition}

See Figure \ref{fig:ex13} for an illustration of $W_{10}$.

\begin{figure} [htpb]
\begin{center}
\begin{tikzpicture}[thick,  every node/.style={circle,draw=black,fill=black!90, inner sep=1.5}, scale=1]
 \node (x0) at (0.0,1.0) [label=above:$x_0$]  [draw=gray, fill=gray]  {};
 \node (x1) at (1.0,1.0) [label=above:$x_1$]  [draw=gray, fill=gray] {};
 \node (x2) at (2.0,1.0) [label=above:$x_2$] {};
 \node (x3) at (3.0,1.0) [label=above:$x_3$]  {};
 \node (x4) at (4.0,1.0) [label=above:$x_4$] {};
 \node (x5) at (5.0,1.0) [label=above:$x_{5}$]  {};
 \node (x6) at (6.0,1.0)  [label=above:$x_6$] {};
 \node (x7) at (7,1.0) [label=above:$x_{7}$] {};
  \node (x8) at (8,1.0) [label=above:$x_{8}$] {};
  \node (x9) at (9,1.0) [label=above:$x_{9}$] {};
  \node (x10) at (10,1.0) [label=above:$x_{0}$] [draw=gray, fill=gray] {};
  \node (x11) at (11,1.0) [label=above:$x_{1}$] [draw=gray, fill=gray] {};
\node(y0) at (0,0)  [label=below:$y_0$] [draw=gray, fill=gray] {};
\node(y1) at (1,0)   [label=below:$y_1$] [draw=gray, fill=gray]  {};
\node(y2) at (2,0)   [label=below:$y_2$] {};
\node(y3) at (3,0)   [label=below:$y_3$] {};
\node(y4) at (4,0)  [label=below:$y_4$]  {};
\node(y5) at (5,0)  [label=below:$y_{5}$]  {};
\node(y6) at (6,0)  [label=below:$y_{6}$]  {};
 \node (y7) at (7,0.0)  [label=below:$y_{7} $]{};
  \node (y8) at (8,0.0)  [label=below:$y_{8} $] {};
  \node (y9) at (9,0.0) [label=below:$y_{9}$] {};
  \node (y10) at (10,0.0) [label=below:$y_{0}$] [draw=gray, fill=gray] {};
   \node (y11) at (11,0.0) [label=below:$y_{1}$] [draw=gray, fill=gray] {};

\path [very thick, draw=color1]
(x0) to (y3)
(y0) to [bend right=15] (y3)
(x0) to [bend left=15] (x3)
(y0) to (x3)
(x0) to (y1)
(x0) to (x1)
(x0) to (y0)
(y0) to (x1)
(y0) to (y1)
(x1) to (y2)
(x1) to (x2)
(x1) to (y1)
(y1) to (x2)
(y1) to (y2)
(x2) to (y5)
(y2) to [bend right=15] (y5)
(x2) to [bend left=15] (x5)
(y2) to (x5)
(x2) to (y3)
(x2) to (x3)
(x2) to (y2)
(y2) to (x3)
(y2) to (y3)
(x3) to (y4)
(x3) to (x4)
(x3) to (y3)
(y3) to (x4)
(y3) to (y4)
(x4) to (y7)
(y4) to [bend right=15] (y7)
(x4) to [bend left=15] (x7)
(y4) to (x7)
(x4) to (y5)
(x4) to (x5)
(x4) to (y4)
(y4) to (x5)
(y4) to (y5)
(x5) to (y6)
(x5) to (x6)
(x5) to (y5)
(y5) to (x6)
(y5) to (y6)
(x6) to (y9)
(y6) to [bend right=15] (y9)
(x6) to [bend left=15] (x9)
(y6) to (x9)
(x6) to (y7)
(x6) to (x7)
(x6) to (y6)
(y6) to (x7)
(y6) to (y7)
(x7) to (y8)
(x7) to (x8)
(x7) to (y7)
(y7) to (x8)
(y7) to (y8)
(x8) to (y11)
(y8) to [bend right=15] (y11)
(x8) to [bend left=15] (x11)
(y8) to (x11)
(x8) to (y9)
(x8) to (x9)
(x8) to (y8)
(y8) to (x9)
(y8) to (y9)
(x9) to (y10)
(x9) to (x10)
(x9) to (y9)
(y9) to (x10)
(y9) to (y10);
	
\end{tikzpicture}
\end{center}
\caption{The graph $W_{10}$; the digraph $W^*_{10}$ is obtained by replacing each edge with a pair of arcs oriented in opposite directions.}
\label{fig:ex13}
\end{figure}
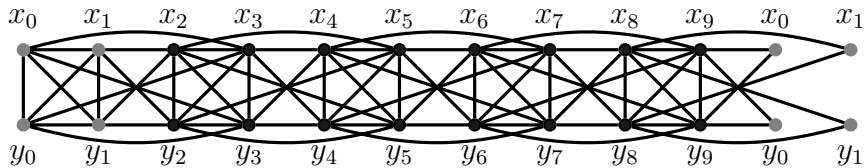

Our goal in this section is to prove the following result.

\begin{lemma}\label{L:reduction}
Let $D$ be a bipartite directed $2$-factor of order $2m \geqslant 14$ comprised of directed cycles of length at least $4$. There is a $D$-factorization of $K^*_{2m}$ if there is a $D$-factorization of $W^*_{2m}$.
\end{lemma}

When $m$ is even we will make use of the following result of Bryant and Danziger \cite{BryDan}.

\begin{lemma}[{\cite[Lemma 7]{BryDan}}]\label{L:BryDan}
For each even $m \geq 8$, there is a factorization of $K_m$ into $\frac{m-4}{2}$ copies of $C_m$ and a copy of $\Circ(m,\{1,3^{\rm e}\})$.
\end{lemma}

We will require an analogue for Lemma~\ref{L:BryDan} for the case where $m$ is odd. To prove this, we state a lemma on decomposition of circulants into hamiltonian cycles. It asserts a special case of a result of Bermond, Favaron, and Mahéo \cite{BerFavMah} on 2-factorizations of Cayley graphs.

\begin{lemma}[\cite{BerFavMah}]\label{L:circFac}
Let $m$ be an integer and let $S$ be a subset of $\{1,\ldots,\lfloor\frac{m-1}{2}\rfloor\}$. Then $\Circ(m,S)$ admits a  $C_m$-factorization if $|S|=2$ and $\gcd(S \cup \{m\})=1$.


\end{lemma}

\begin{lemma}~\label{L:BryDanAnalogue}
For each odd $m \geq 7$, the graph $K_m$ admits a decomposition into $\frac{m-5}{2}$ copies of $C_m$ and one copy of $\Circ(m,\{1,2\})$.
\end{lemma}

\begin{proof}
The graph $\Circ(m,\{1,\ldots,\frac{m-1}{2}\})$ is a copy of $K_m$. If $m=7$, then $\Circ(7,\{1,2, 3\})$ has a decomposition $\F= \{\Circ(7, \{1,2\}), \Circ(7, \{3\})\}$. If $m \geqslant 9$, then $\Circ(m,\{1,\ldots,\frac{m-1}{2}\})$ has a decomposition $\F$ given by
\[
\begin{array}{ll}
    \bigl\{\Circ(m,S): S \in \{\{1,2\},\{3,4\},\ldots,\{\tfrac{m-3}{2},\tfrac{m-1}{2}\}\}\bigr\} &
\hbox{if $m \equiv 1 \mod{4}$}; \\
    \bigl\{\Circ(m,S): S \in \{\{1,2\},\{3, 5\}, \{4\},\{6,7\},\{8,9\}, \ldots,\{\tfrac{m-3}{2},\tfrac{m-1}{2}\}\}\bigr\} & \hbox{if $m \equiv 3 \mod{4}$.}
  \end{array}
\]

\noindent Clearly, $\Circ(m, \{4\})$ is a copy of $C_m$ when $m$ is odd and $\Circ(7, \{3\})$ is a copy of $C_7$. Therefore, in each case it can be seen that each subgraph in $\F$ other than $\Circ(m,\{1,2\})$ admits a $C_m$-factorization by using Lemma~\ref{L:circFac}. Taking the union of these $C_m$-factorizations together with $\Circ(m,\{1,2\})$ completes the proof. \end{proof}

Using Lemmas~\ref{L:BryDan} and \ref{L:BryDanAnalogue} we can complete our proof of Lemma~\ref{L:reduction}.

\begin{proof}[\textup{\textbf{Proof of Lemma~\ref{L:reduction}}}]

By Lemmas~\ref{L:BryDan} and \ref{L:BryDanAnalogue}, there is a decomposition $\{G\} \cup \mathcal{C}$ of $K_m$ where $\mathcal{C}$ is a set of directed cycles of length $m$, $G$ is a copy of $\Circ(m,\{1,3^{\rm e}\})$ if $m$ is even, and $G$ is a copy of $\Circ(m,\{1,2\})$ if $m$ is odd. Since $K_m \wr K_2$ is isomorphic to $K_{2m}$, we have that $\F$ is a decomposition of $K_{2m}$ where
\[\mathcal{F}=\{G \wr K_2\} \cup \{C \wr E_2:C \in \C\}.\]

Noting $(G \wr K_2)^*$ is a copy of $W^*_{2m}$, we see that $\F^*=\{F^*:F \in \F\}$ is a decomposition of $K^*_{2m}$ into copies of $(C_m \wr E_2)^*$ and one copy of $W^*_{2m}$. By Lemma~\ref{L:Hag}, $(C \wr E_2)^*$ has a $D$-factorization $\D_C$ for each $C \in \C$. Thus, if $W^*_{2m}$ has a $D$-factorization $\D'$, then $\D' \cup \{\D_C:C \in \C\}$ will be a $D$-factorization of $K^*_{2m}$.
\end{proof}

In summary Lemma~\ref{L:reduction} implies that, to prove Theorem~\ref{thm:twotab}, it suffices to construct a $(\vec{C}_{t_1}, \vec{C}_{t_2})$-factorization of $W^*_{t_1+t_2}$ when $t_1 \in \{4,6\}$, $t_2$ is even, and $t_1+t_2\geqslant 14$.

\section{Main construction} \label{S:main}

Throughout this section we take $t_1$ and $q$ to be fixed integers such that $t_1 \in \{4,6\}$ and
\begin{equation}\label{E:qChoice}
q \in \left\{
  \begin{array}{ll}
    \{10,14,16,20\} & \hbox{if $t_1=4$}; \\
    \{14,16,18,20\} & \hbox{if $t_1=6$.}
  \end{array}
\right.
\end{equation}

Our goal will be to show that a $(\vec{C}_{t_1}, \vec{C}_{q+8k})$-factorization of $W^*_{t_1+q+8k}$ exists for each nonnegative integer $k$. Lemma \ref{L:reduction} will establish our main theorem for all pairs $(t_1,q)$ except those in $\{(6,8), (6, 12), (4,12), (6,10)\}$. We will then deal with these special cases in Appendix \ref{sub:spec}.

\begin{notation}
\label{not:ver} \rm
Throughout the remainder of the paper we shall assume that, in the definition of $W^*_{2m}$ given in Definition~\ref{defn:W}, the copy of $K_2$ has vertex set $\{x,y\}$. Further, we will abbreviate vertices $(a,x)$ to $x_a$ and vertices $(b,y)$ to $y_b$ so that
\begin{center}
$V(W^*_{2m})=\{x_a, y_b\, : \ a,b \in \mathds{Z}_m\}$.
\end{center}
\end{notation}

In addition, we define the following permutations of $V(W^*_{2m})$.

\begin{definition}
\label{def:rho}
For each even integer $j$, we will take $\rho^j$ to be the permutation of $V(W^*_{2m})$ defined by $\rho^j(x_i)=x_{i+j}$ and $\rho^j(y_i)=y_{i+j}$, with subscript addition performed modulo $m$. For a dipath $P=v_0v_1\cdots v_t$ of $W^*_{2m}$, we let $\rho^j(P)=\rho^j(v_0)\rho^j(v_1)\rho^j(v_2)\cdots \rho^j(v_t)$. We refer to a dipath $\rho^j(P)$ as a \textit{translation} of $P$, and note that $\rho^j(P)$ is also a dipath of $W^*_{2m}$ since $j$ is even.
\end{definition}

Each of the factors in the directed 2-factorizations we desire will be created from what we call a \emph{($t_1, q$)-base tuple} $(X, Q, R, S, T)$ where $X$ is a directed $t_1$-cycle and $Q$, $R$, $S$ and $T$ are dipaths of various lengths. We will define ($t_1, q$)-base tuples formally in Definition \ref{defn:base2.2} below, but first we give an informal overview of how they will be used. For a given nonnegative integer $k$, from each ($t_1, q$)-base tuple $(X, Q, R, S, T)$,  we will construct a $(\vec{C}_{t_1}, \vec{C}_{q+8k})$-factor which is a union of the following pieces:

\begin{itemize}
    \item
a directed $t_1$-cycle $X$; 
    \item two dipaths $I_0$ and $I_1$ formed as the concatenation of  $k$ translations of $R$ and $S$, respectively;
    \item two dipaths $Q$ and $R$ such that $s(Q)=t(I_1)$, $t(Q)=s(I_0)$, $s(R)=t(I_0)$, and $t(R)=s(I_1)$.
\end{itemize}

The union of $Q$, $R$, $I_0$, and $I_1$ will form a directed $(q+8k)$-cycle that is disjoint from $X$. A schematic picture of this construction is given in Figure \ref{fig:schema}. Since each ($t_1, q$)-base tuple gives us factors of infinitely many orders, this approach will allow us to reduce our problem to finding only eight sets of ($t_1, q$)-base tuples (one for each possible choice of $(t_1,q)$).

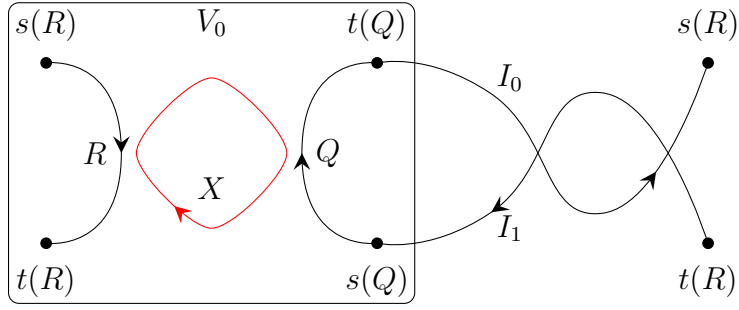
\begin{figure}[H]
\begin{center}
\begin{tikzpicture}

\node[label=below:$t(R)$] (A) at (0,-0.2) {};
\node[label=right:$R$] (R) at (0.2,1) {};
\node[label=above:$s(R)$] (B) at (0,2.2){};
\node[label=above:$X$] (X) at (2.2,0.1){};
\node[label=below:$s(Q)$] (C) at (4.4,-0.2) {};
\node[label=left:$Q$] (Q) at (4.2,1){};
\node[label=above:$t(Q)$] (D) at (4.4,2.2){};
\node[label=below:$t(R)$] (A1) at (8.8,-0.2) {};
\node[label=above:$s(R)$] (B1) at (8.8,2.2){};
\node[label=right:$I_1$] (Q) at (5.7,0){};
\node[label=right:$I_0$] (Q) at (5.7,2){};
\node[label=below:$V_0$] (V0) at (2.2,3.2){};

\draw[rounded corners] (-0.5, -1) rectangle (4.9, 3) {};

\draw plot [smooth, tension=2, very thick] coordinates {(B) (1,1)   (A)}
[arrow inside={end=stealth,opt={scale=2}}{0.5}];

\draw [red] plot [smooth cycle] coordinates {(1.2,1) (2.2,2) (3.2,1) (2.2,0)}
[arrow inside={end=stealth,opt={red,scale=2}}{0.6}];

\draw plot [smooth, tension=2] coordinates {(C) (3.4,1)   (D)}
[arrow inside={end=stealth,opt={scale=2}}{0.5}];

\draw plot [smooth, tension=1] coordinates {(D) (5.9,1.8)  (7.4, 0.2)  (B1)}
[arrow inside={end=stealth,opt={scale=2}}{0.75}];

\draw plot [smooth, tension=1] coordinates {(A1) (7.4,1.8)  (5.9, 0.2)  (C)}
[arrow inside={end=stealth,opt={scale=2}}{0.75}];

\draw [fill=black] (A) circle (2pt);
\draw [fill=black] (B) circle (2pt);
\draw [fill=black] (C) circle (2pt);
\draw [fill=black] (D) circle (2pt);
\draw [fill=black] (A1) circle (2pt);
\draw [fill=black] (B1) circle (2pt);
\end{tikzpicture}
\end{center}
\caption{ A schema of the construction of a directed 2-factor of $W^*_{2m}$ from a ($t_1, q$)-base tuple $(X, Q, R, S, T)$ with directed $t_1$-cycle $X$ drawn in red.}
\label{fig:schema}
\end{figure}

\begin{notation} \rm
\label{not:sets}
For the remainder of this section, it will be useful to set $p=\frac{1}{2}(t_1+q)$, so that $2p$ is the smallest of the orders of the directed 2-factorizations we desire. For each non-negative integer $k$, we also define certain subsets of the vertex set of $W^*_{p+4k}$ as follows:
\begin{align*}
  V_0 &= \{x_i, y_i\ : \ 0\leqslant i \leqslant p+1\}; \\
  V^{\dag}_0 &=\{x_i, y_i\ : \ 2\leqslant i \leqslant p-1\}; \\
  V_j &=\{x_{i},y_{i}\ : \ p+4j-4 \leqslant i \leqslant p+4j+1\} \text{ for each $j \in \{1, 2, \ldots, k\}$}.
\end{align*}
\end{notation}

Observe that $V_0=V(W^*_{p+4k})$ if $k=0$, $|V_0|=2(p+2)$ otherwise, and $|V_j|=12$ for each $j \in \{1,2, \ldots, k\}$. Also, for all integers $i$ and $j$ with $0 \leq i < j \leq k$,
\begin{equation}\label{E:intersections}
V_i \cap V_j=
\left\{
  \begin{array}{ll}
    \{x_{{p-1}}, x_{{p}}, y_{{p-1}}, y_{{p}}\} & \hbox{if $(i,j)=(0,1)$;} \\
    \{x_0, x_1, y_0, y_1\} & \hbox{if $(i,j)=(0,k)$;} \\
    \{x_{p+4i}, x_{p+4i+1}, y_{p+4i}, y_{p+4i+1}\} & \hbox{if $j=i+1$ and $i \geq 1$;} \\
    \emptyset & \hbox{otherwise.}
  \end{array}
\right.
\end{equation}

Let $k \geq 2$ be an integer and let $A=(P^1,\ldots,P^k)$ be a sequence of dipaths. We say that $A$ \emph{concatenates} if $t(P^i)=s(P^{i+1})$ for each $i \in \{1,\ldots,k-1\}$ and, aside from this, no vertex is in more than one dipath in the sequence. In this case we call the dipath $P^1\cup\cdots\cup P^k$ the \emph{concatenation} of $A$.  Similarly we say that $A$ \emph{cyclically concatenates} if $t(P^k)=s(P^1)$, $t(P^i)=s(P^{i+1})$ for each $i \in \{1,\ldots,k-1\}$ and, aside from this, no vertex is in more than one dipath in the sequence. In this case we call the directed cycle $P^1\cup\cdots\cup P^k$ the \emph{cyclic concatenation} of $A$.

\begin{definition} \rm
\label{defn:base2.2}
The 5-tuple $(X, Q, R, S, T)$ is a  \textit{$(t_1,q)$-base tuple} if $X$ is a directed $t_1$-cycle of $W^*_{t_1+q+24}$ and $Q$, $R$, $S$ and $T$ are dipaths of $W^*_{t_1+q+24}$ with the following properties.
\begin{enumerate} [label=\textbf{B\arabic*}]
\item $V(X) \subseteq V^{\dag}_0$; $V(Q) \subseteq V_0 \setminus \{x_0,y_0,x_1,y_1\}$; $V(R) \subseteq V_0 \setminus \{x_p,y_p,x_{p+1},y_{p+1}\}$; $V(S), V(T) \subseteq V_1$;
\item  $\len(Q)+\len(R)=q$ and $\len(S)+\len(T)=8$;
\item $X$, $Q$, and $R$ are pairwise vertex-disjoint;
\item  $(Q,\rho^p(R))$ cyclically concatenates;
\item $(T,Q,S)$ and $(\rho^{-p-4}(S),R,\rho^{-p-4}(T))$ concatenate;
\item $(S,\rho^4(S))$ and $(T,\rho^4(T))$ concatenate. Further, the concatenations of $(S,\rho^4(S))$ and $(T,\rho^4(T))$ are vertex-disjoint.
\end{enumerate}
\end{definition}

In Definition~\ref{defn:base2.2}, the choice of $t_1+q+24$ to be the order of the host graph is somewhat arbitrary. We could equivalently have chosen any other order large enough to ensure that the translations mentioned in B4, B5, and B6 do not contain vertices of $V^{\dag}_0$. Our next lemma describes how we can use a ($t_1, q$)-base tuple to obtain a directed $2$-factor.

\begin{lemma}
\label{lem:red.2}
Let  $(X, Q, R, S, T)$ be a $(t_1,q)$-base tuple and let $k$ be a nonnegative integer. In the host graph $W^*_{t_1+q+8k}$, denote $S^j=\rho^{4(j-1)}(S)$ and $T^j=\rho^{4(j-1)}(T)$ for each $j \in \{1,2,\ldots, k\}$ and let
\[A=(Q,S^1,S^2,\ldots, S^{k},R, T^{k},T^{k-1},\ldots, T^1).\]
Then $A$ cyclically concatenates and, furthermore, $X$ and the cyclic concatenation of $A$ form the cycles of a $(\vec{C}_{t_1}, \vec{C}_{q+8k})$-factor of $W^*_{t_1+q+8k}$.
\end{lemma}

\begin{proof}
If $A$ cyclically concatenates, then its cyclic concatenation is a directed cycle of length $q+8k$ by B2. So it suffices to prove that $A$ does indeed cyclically concatenate and that its cyclic concatenation is vertex-disjoint from $X$.

\noindent \textbf{Case 1.} Suppose that $k=0$. Then our host graph is $W^*_{t_1+q}=W^*_{2p}$ and $A=(Q,R)$. In $W^*_{2p}$, we have that $\rho^p$ is the identity permutation and so B4 implies that $A$ cyclically concatenates. Further, B3 implies that the cyclic concatenation of $A$ is vertex-disjoint from $X$.

\noindent \textbf{Case 2.} Suppose that $k=1$. Then our host graph is $W^*_{t_1+q+8}=W^*_{2(p+4)}$ and $A=(Q,S^1,R,T^1)$. In $W^*_{2(p+4)}$, we have that $\rho^{-p-4}$ is the identity permutation and so B5 implies that $(T^1,Q,S^1)$ and $(S^1,R,T^1)$ concatenate. So, because $Q$ and $R$ are vertex-disjoint by B3, we have that $A$ cyclically concatenates. By B3 $Q$ and $R$ are both vertex-disjoint from $X$. By B1 and \eqref{E:intersections}, both $S$ and $T$ are vertex-disjoint from $X$. So the cyclic concatenation of $A$ is vertex-disjoint from $X$.

\noindent \textbf{Case 3.} Suppose that $k \geq 2$. Our host graph is $W^*_{t_1+q+8k}=W^*_{2(p+4k)}$. In $W^*_{2(p+4k)}$, we have that $\rho^{4(k-1)}=\rho^{-p-4}$ and so B5 implies that $(T^1,Q,S^1)$ and $(S^k,R,T^k)$ concatenate. By B6, we have that $(S^i,S^{i+1})$ and $(T^i, T^{i+1})$ both concatenate for each $i \in \{1,\ldots,k-1\}$. Also by B6, for each $i \in \{1,\ldots,k-1\}$, we have that $S^i$ is vertex-disjoint from $T^i$ and $T^{i+1}$ and that $T^i$ is vertex-disjoint from $S^i$ and $S^{i+1}$. Note that B1 implies that $V(S^i)$ and $V(T^i)$ are subsets of $V_i$ for each $i \in \{1,\ldots,k\}$. Thus we can conclude that $A$ cyclically concatenates because B1 and \eqref{E:intersections} imply that all the remaining vertex-disjointness conditions are met. Further, using B3 together with B1 and \eqref{E:intersections}, we have that the cyclic concatenation of $A$ is vertex-disjoint from $X$. \end{proof}


For digraphs $G$ and $H$, we use the notation $G \cong H$ to indicate that $G$ and $H$ are isomorphic. Furthermore, for digraphs $G_1$, $G_2$, $H_1$ and $H_2$, we write $(G_1,G_2) \cong (H_1,H_2)$ to indicate that $G_1 \cong H_1$, $G_2 \cong H_2$, and $G_1 \cup G_2 \cong H_1 \cup H_2$. Note that if $(G_1,G_2) \cong (H_1,H_2)$ and $H_1$ is arc-disjoint from $H_2$, then we must have that $G_1$ is arc-disjoint from $G_2$ because
\[|A(G_1 \cup G_2)| = |A(H_1 \cup H_2)| = |A(H_1)| + |A(H_2)| = |A(G_1)| + |A(G_2)|.\]

The next lemma gives conditions under which the $2$-factors arising from a number of ($t_1, q$)-base tuples form a directed 2-factorization.

\begin{lemma}
\label{lem:red2}
Let $r=9$ if $t_1+q\equiv 2 \mod{4}$, and let $r=7$ if $t_1+q \equiv 0 \mod{4}$. For each $a \in \{0,\ldots,r-1\}$  let $(X_a, Q_a, R_a, S_a, T_a)$ be a $(t_1,q)$-base tuple and let $\hat{F}_a$ be the corresponding $(\vec{C}_{t_1}, \vec{C}_{q+16})$-factor of $W^*_{t_1+q+16}$. If the digraphs in $\hat{\F}=\hat{F}_0,\ldots,\hat{F}_{r-1}$ are pairwise arc-disjoint, then $W^*_{t_1+q+8k}$ admits a $(\vec{C}_{t_1}, \vec{C}_{q+8k})$-factorization for each positive integer $k$. If, in addition, each of the dipaths $Q_0,\ldots,Q_{r-1}$ is arc-disjoint from each of the dipaths $\rho^p(R_0),\ldots,\rho^p (R_{r-1})$, then $W^*_{t_1 +q}$ admits a $(\vec{C}_{t_1}, \vec{C}_{q})$-factorization.
\end{lemma}

\begin{proof}
Fix a nonnegative integer $k$ and suppose that the digraphs in $\hat{\F}=\hat{F}_0,\ldots,\hat{F}_{r-1}$ are pairwise arc-disjoint and that, if $k=0$, each of the paths $Q_0,\ldots,Q_{r-1}$ is arc-disjoint from each of the paths $\rho^p(R_0),\ldots,\rho^p (R_{r-1})$. For each $a \in \{0,\ldots, r-1\}$, let $F_a$ be the  $(\vec{C}_{t_1}, \vec{C}_{q+8k})$-factor of $W^*_{t_1+q+8k}$ constructed using the $(t_1,q)$-base tuple $(X_a, Q_a, R_a, S_a, T_a)$. Using the notation of Lemma~\ref{lem:red.2} relative to the host graph $W^*_{t_1+q+8k}$,  $\hat{F}_a$ is the union of $X_a$ and the cyclic concatenation of
\[(Q_a,S_a^1,S_a^2,S_a^2,\ldots, S_a^{k},R_a,T_a^{k},T_a^{k-1},\ldots, T_a^1)\]

Again using the notation of Lemma~\ref{lem:red.2}, but adding hats to indicate that the notation is relative to the host graph $W^*_{t_1+q+16}$, we let $\hat{F}_a$ be the union of $\hat{X}_a$ and the cyclic concatenation of

\[(\hat{Q}_a,\hat{S}_a^1,\hat{S}_a^2,\hat{R}_a,\hat{T}_a^{2},\hat{T}_a^1)\]

We must show that $\mathcal{F}=\{F_0, F_1, \ldots, F_{r-1}\}$ is a $(\vec{C}_{t_1}, \vec{C}_{q+8k})$-factorization of  $W^*_{t_1+q+8k}$. Observe that $\sum_{a=0}^{r-1}|A(F_a)|=|A(W^*_{t_1+q+8k})|$. Therefore, it suffices to show that the directed 2-factors in $\mathcal{F}=\{F_0, F_1, \ldots, F_{r-1}\}$ are pairwise arc-disjoint. This follows immediately from the hypothesis if $k=2$, so we can assume otherwise.

If $k=1$ then, for each $a \in \{0,\ldots,r-1\}$, $F_a$ is obtained from $\hat{F}_a$ by associating the vertices $x_{p+4},y_{p+4},\ldots,x_{p+7},y_{p+7}$ with, respectively, the vertices $x_{p},y_{p},\ldots,x_{p+3},y_{p+3}$. In this association process, by their definitions, the subdigraphs $\hat{S}_a^1$ and $\hat{S}_a^2$ of $\hat{F}_a$ both map to the subdigraph $S_a^1$ of $F_a$, and the subdigraphs $\hat{T}_a^1$ and $\hat{T}_a^2$ of $\hat{F}_a$ both map to the subdigraph $T_a^1$ of $F_a$. Together with B1. This ensures that no two arcs in different factors of $\hat{\F}$ are mapped onto the same arc. Thus, because the factors in $\hat{\F}$ are pairwise arc-disjoint, the factors in $\mathcal{F}$ are pairwise arc-disjoint.

If $k=0$ then, for each $a \in \{0,\ldots,r-1\}$, $F_a$ is obtained from $\hat{F}_a$ by deleting the arcs in $S_a^1 \cup S_a^2 \cup T_a^1 \cup T_a^2$, deleting the vertices $x_{p+2},y_{p+2},\ldots,x_{p+7},y_{p+7}$, and associating the vertices $x_{p},y_{p},x_{p+1},y_{p+1}$ with, respectively, the vertices $x_{0},y_{0},x_{1},y_{1}$. Our additional assumption in the case $k=0$, together with B1, ensures that no two arcs in different factors of $\hat{\F}$ are mapped onto the same arc in this association process. Thus, because the factors in $\hat{\F}$ are pairwise arc-disjoint, the factors in $\mathcal{F}$ are pairwise arc-disjoint.



Lastly, we consider the case $k \geq 3$. For each $a \in \{0,\ldots,r-1\}$, we let $Y_a=X_a \cup Q_a \cup R_a$ and, for each $i \in \{1,\ldots,k\}$, $U^i_a=S^i_a \cup T^i_a$. Also, we let $\hat{Y}=\hat{X}_a \cup \hat{Q}_a \cup \hat{R}_a$ and, for $j \in \{1,2\}$, we let $\hat{U}^j_a=\hat{S}^j_a \cup \hat{T}^j_a$. Let $\{V_0,\ldots,V_k\}$ be the partition of $V(W^*_{t_1+q+8k})$ defined in Notation~\ref{not:sets}. By B1 and \eqref{E:intersections} we have, for each $a \in \{0,\ldots,r-1\}$:

\begin{itemize}
    \item[(i)]
$V(Y_a) \subseteq V_0$; and
    \item[(ii)]
$V(U^i_a) \subseteq V_i$ for each $i \in \{1,\ldots,k\}$.
\end{itemize}

Let $g$ and $\ell$ be distinct elements of $\{0,\ldots,r-1\}$. We complete the proof by showing that $F_\ell$ is arc-disjoint from $F_g$. We do this by first showing that $Y_\ell$ is arc-disjoint from $F_g$. Then, for each $j \in \{1,\ldots,k\}$, we show that $U^j_\ell$ is arc-disjoint from $F_g$.



\noindent {\bf Case 1: $\bm{Y_\ell}$}. Using (i) and (ii) we have that $Y_\ell$ is vertex-disjoint from $\bigcup_{i=2}^{k-1}U^i_g$ because $V_0$ is vertex-disjoint from $V_2,\ldots,V_{k-1}$. Now $Y_\ell \cup U^1_\ell$ is arc-disjoint from $Y_g \cup U^1_g$ because $(Y_\ell \cup U^1_\ell,Y_g \cup U^1_g) \cong (\hat{Y}_\ell \cup \hat{U}^1_\ell,\hat{Y}_g \cup \hat{U}^1_g)$ and $\hat{Y}_\ell \cup \hat{U}^1_\ell$ is arc-disjoint from $\hat{Y}_g \cup \hat{U}^1_g$.  Similarly, $U^k_\ell \cup Y_\ell$ is arc-disjoint from $U^k_g \cup Y_g$ because $(U^k_\ell \cup Y_\ell,U^k_g \cup Y_g) \cong (\hat{U}^2_\ell \cup \hat{Y}_\ell,\hat{U}^2_g \cup \hat{Y}_g)$ and $\hat{U}^2_\ell \cup \hat{Y}_\ell$ is arc-disjoint from $\hat{U}^2_g \cup \hat{Y}_g$. So, in particular, $Y_\ell$ is arc-disjoint from $U^k_g \cup Y_g \cup U^1_g$ and hence from $F_g$.  \smallskip



\noindent {\bf Case 2: $\bm{U^j_\ell}$ where $\bm{j \in \{2,\ldots,k-1\}}$.} Let $\mathds{I}=\{0,1,\ldots, j-2, j+2, j+3, \ldots, k\}$. From (i) and (ii), the digraph $U^j_\ell$ is vertex-disjoint from $Y_g \cup \bigcup_{i \in \mathds{I} \setminus \{0\}}U^i_g$ because $V_j$ is vertex-disjoint from $\bigcup_{i \in \mathds{I}}V_i$. Now, $U^{j-1}_\ell \cup U^{j}_\ell$ is arc-disjoint from $U^{j-1}_g \cup U^{j}_g$ because
\[(U^{j-1}_\ell \cup U^{j}_\ell,U^{j-1}_g \cup U^{j}_g) \cong (U^{1}_\ell \cup U^{2}_\ell,U^{1}_g \cup U^{2}_g) \cong (\hat{U}^{1}_\ell \cup \hat{U}^{2}_\ell,\hat{U}^{1}_g \cup \hat{U}^{2}_g)\]
and $\hat{U}^{1}_\ell \cup \hat{U}^{2}_\ell$ is arc-disjoint from $\hat{U}^{1}_g \cup \hat{U}^{2}_g$. Similarly, $U^{j}_\ell \cup U^{j+1}_\ell$ is arc-disjoint from $U^{j}_g \cup U^{j+1}_g$ because
$(U^{j}_\ell \cup U^{j+1}_\ell,U^{j}_g \cup U^{j+1}_g) \cong (\hat{U}^{1}_\ell \cup \hat{U}^{2}_\ell,\hat{U}^{1}_g \cup \hat{U}^{2}_g)$. So, in particular, $U_j$ is arc-disjoint from $U^{j-1}_g \cup U^{j}_g \cup U^{j+1}_g$ and hence from $F_g$.\smallskip

\noindent {\bf Case 3: $\bm{U^1_\ell}$ and \bm{$U^k_\ell$}}. From (i) and (ii), it follows that $U^1_\ell$ is vertex-disjoint from $\bigcup_{i=3}^kU^i_g$ because $V_1$ is vertex-disjoint from $V_3,\ldots,V_{k}$. Likewise, $U^k_\ell$ is arc-disjoint from $\bigcup_{i=1}^{k-2}U^i_g$ because $V_k$ is vertex-disjoint from $V_1,\ldots,V_{k-2}$. In Case 1, we saw that $U^1_\ell$ and $U^k_\ell$ are both arc-disjoint from $Y_g$. In Case 2, (with $j=2$) we saw that $U^1_\ell$ is arc-disjoint from $U^2_g$. In Case 2, (with $j=k-1$) we also saw that $U^k_\ell$ is arc-disjoint from $U^{k-1}_g$. So both $U^1_\ell$ and $U^k_\ell$ are arc-disjoint from $F_g$.\smallskip

In summary, we have demonstrated that $F_\ell$ is arc-disjoint from $F_g$ for distinct $\ell$ and $g$. Therefore, the given set of $r$ ($t_1, q$)-base tuples gives rise to the desired directed 2-factorization of $W^*_{2m}$.
\end{proof}




We now conclude this section with the proof of this paper's main result, namely the proof of Theorem \ref{thm:twotab}.


\begin{proof}[\textup{\textbf{Proof of Theorem~\ref{thm:twotab}}}]

We show that $K^*_{2m}$ admits a $(\vec{C}_{t_1}, \vec{C}_{t_2})$-factorization when $t_1+t_2=2m$, $t_1 \in \{4,6\}$, and $t_1+t_2 \geq 14$. Lemma \ref{L:reduction} implies that it suffices to find a $(\vec{C}_{t_1}, \vec{C}_{t_2})$-factorization of $W^*_{2m}$. For the special cases where $(t_1, t_2) \in \{(4,12),  (6,8), (6,10), (6,12)\}$ we give a $(\vec{C}_{t_1}, \vec{C}_{t_2})$-factorization of $W^*_{2m}$ in Appendix \ref{sub:spec}. Otherwise, we have that $t_2=q+8k$ for some $q$ satisfying \eqref{E:qChoice} and nonnegative integer $k$. Let $r=9$ if $m$ is odd and $r=7$ if $m$ is even. To construct a  $(\vec{C}_{t_1}, \vec{C}_{t_2})$-factorization of $W^*_{2m}$, it suffices to construct $r$ ($t_1, q$)-base tuples satisfying the hypothesis of Lemma~\ref{lem:red2}.

\begin{enumerate}
\item If $m$ is odd, then $(t_1, q) \in \{(4,10), (4, 14), (6, 16), (6,20)\}$. Appendix \ref{sub:case1} gives a set of nine ($t_1, q$)-base tuples satisfying the hypothesis of Lemma \ref{lem:red2} for each of these choices of ($t_1, q$).
\item If $m$ is even, then $(t_1, q) \in \{(4,16), (4, 20), (6, 14), (6,18)\}$. Appendix \ref{sub:case2} gives a set of seven ($t_1, q$)-base tuples satisfying the hypothesis of Lemma \ref{lem:red2} for each of these choices of ($t_1, q$).
\end{enumerate}

In conclusion, the digraph $W^*_{2m}$ admits a $(\vec{C}_{t_1}, \vec{C}_{t_2})$-factorization when $t_1+t_2=2m$, $t_1 \in \{4,6\}$, and $t_1+t_2 \geq 14$. It follows that the OP$^*(t_1, t_2)$ has a solution for all applicable $t_1$ and $t_2$ values.
\end{proof}

The $(t_1,q)$-tuples presented in Appendices \ref{sub:case1} and \ref{sub:case2}  were constructed by hand with the assistance of a computer. For example, in many cases, we first used a computer to obtain an exhaustive list of all possible sets of dipaths $\{S_0,T_0,\ldots,S_{r-1},T_{r-1}\}$ and  $\{Q_0,R_0,\ldots, Q_{r-1},R_{r-1}\}$. The process of fitting these together, making adjustments if necessary, and completing the tuples such that they give rise to the desired  directed 2-factorization, however, was largely accomplished by hand.

 \subsection*{Acknowledgements}
The first author was supported by the Australian Research Council (grants DP220102212 and DP240101048). The second author was supported by the Natural Sciences and Engineering Research Council of Canada (NSERC) Post Graduate Scholarship program and the NSERC Michael Smith Foreign Study Supplement program.

\appendix
 
\section{Supplementary material for the proof of Theorem~\ref{thm:twotab}}
\label{A:A} 

 \subsection{The case $t_1+q \equiv 2\ (\textrm{mod}\ 4)$}
  \label{sub:case1}
\noindent {\bf Case 1}: $t_1=4$ and $q=10$. See Figures \ref{fig:case1.1}--\ref{fig:case1.9} in Appendix \ref{A:B1}.

\begin{center}
$
\begin{array}{p{0cm} l l l l l }
&X_0=x_4\,x_{3}\,y_{5}\,y_{3}\,x_{4};  &&X_3=\,y_{5}\,y_{4}\,y_{2}\,x_{4}\,y_{5};  &&X_6=\,x_{4}\,y_{2}\,y_{4}\,x_{2}\,x_{4}; \\
&Q_0=\,y_{7}\,x_{5}\,y_{6}\,y_{4}\,x_{6}\,x_{7};  &&Q_3=\,x_{8}\,y_{7}\,x_{7}\,x_{5}\,x_{6}\,y_{6}\,y_{8};  &&Q_6=\,x_{8}\,x_{6}\,y_{7}\,y_{5}\,y_{6}\,x_{5}\,x_{7};\\
&R_0=\,x_{0}\,y_{2}\,x_{1}\,x_{2}\,y_{1}\,y_{0}; &&R_3=\,y_{1}\,y_{3}\,x_{2}\,x_{3}\,x_{1}; &&R_6=\,x_{0}\,y_{1}\,x_{3}\,y_{3}\,x_{1}; \\ 
&S_0=\,x_{7}\,y_{9}\,x_{9}\,x_{8}\,y_{10}\,x_{11};  &&S_3=\,y_{8}\,x_{9}\,y_{11}\,x_{11}\,x_{10}\,y_{12}; & &S_6=\,x_{7}\,y_{8}\,x_{10}\,x_{11}; \\ 
&T_0=\,y_{11}\,x_{10}\,y_{8}\,y_{7}; &&T_3=\,x_{12}\,y_{10}\,y_{9}\,x_{8};  &&T_6=\,x_{12}\,y_{11}\,x_{9}\,y_{9}\,y_{10}\,x_{8}; \\ 
&&&&&\\
&X_1=\,y_{5}\,x_{3}\,y_{4}\,x_{5}\,y_{5};&&X_4=\,x_{5}\,x_{3}\,x_{4}\,y_{3}\,x_{5}; &&X_7=\,y_{5}\,x_{5}\,y_{4}\,y_{3}\,y_{5}; \\ 
&Q_1=\,y_{7}\,x_{6}\,x_{4}\,y_{6}\,x_{8}; &&Q_4=\,x_{7}\,x_{6}\,y_{4}\,y_{6}\,y_{5}\,y_{7};  &&Q_7=\,y_{8}\,y_{6}\,x_{4}\,x_{6}\,x_{8}; \\
\end{array}$
\end{center}

\noindent {\bf Case 2}:  $t_1=4$ and $q=14$. See Figures \ref{fig:case2.1}--\ref{fig:case2.9}  in Appendix \ref{A:B1}.

\begin{center}
$\begin{array}{p{0cm} l l}
&X_0=x_3\,y_{3}\,x_{5}\,y_{4}\,x_{3}; &S_0=\,x_{10}\,y_{10}\,y_{9}\,x_{11}\,y_{12}\,x_{14}; \\ 
&Q_0=\,x_{9}\,x_{8}\,x_{6}\,y_{5}\,x_{7}\,y_{6}\,y_{7}\,y_{8}\,x_{10}; &T_0=\,x_{13}\,x_{12}\,y_{11}\,x_{9}; \\ 
&R_0=\,x_{1}\,y_{1}\,y_{0}\,x_{2}\,x_{4}\,y_{2}\,x_{0};& \\ 
&&\\
&X_1=\,x_{3}\,x_{5}\,x_{4}\,y_{5}\,x_{3};&S_1=\,y_{9}\,x_{9}\,x_{11}\,y_{11}\,y_{13}; \\ 
&Q_1=\,x_{10}\,x_{8}\,y_{10}\,y_{8}\,x_{6}\,y_{6}\,x_{7}\,y_{7}\,y_{9}; &T_1=\,x_{14}\,x_{12}\,y_{14}\,y_{12}\,x_{10}; \\ 
&R_1=\,y_{0}\,x_{0}\,x_{2}\,y_{4}\,y_{2}\,y_{3}\,x_{1}; &\\ 
&&\\
 &X_2=\,y_{3}\,y_{5}\,x_{4}\,x_{5}\,y_{3}; &S_2=\,x_{10}\,y_{12}\,y_{14}\,x_{12}\,x_{14}; \\ 
&Q_2=\,y_{9}\,y_{7}\,y_{6}\,x_{6}\,x_{7}\,y_{8}\,y_{10}\,x_{8}\,x_{10};&T_2=\,y_{13}\,y_{11}\,x_{11}\,x_{9}\,y_{9}; \\ 
&R_2=\,x_{1}\,x_{3}\,y_{2}\,y_{4}\,x_{2}\,x_{0}\,y_{0}; & \\ 
\end{array}$
\end{center}

\begin{center}
$\begin{array}{p{0cm} l l}
&X_3=\,y_{3}\,x_{4}\,x_{6}\,y_{4}\,y_{3}; &S_3=\,x_{9}\,y_{10}\,x_{12}\,y_{12}\,y_{11}\,x_{13}; \\ 
&Q_3=\,y_{9}\,x_{7}\,x_{5}\,y_{5}\,y_{6}\,x_{8}\,y_{8}\,y_{7}\,x_{9}; &T_3=\,y_{13}\,x_{11}\,x_{10}\,y_{9}; \\ 
&R_3=\,x_{0}\,y_{1}\,y_{2}\,x_{3}\,x_{1}\,x_{2}\,y_{0}; \\
&&\\
&X_4=\,y_{3}\,y_{4}\,y_{6}\,x_{4}\,y_{3}; &S_4=\,y_{9}\,x_{10}\,y_{11}\,x_{12}\,y_{13}; \\ 
&Q_4=\,x_{9}\,y_{8}\,x_{7}\,y_{5}\,y_{7}\,x_{5}\,x_{6}\,x_{8}\,y_{9}; &T_4=\,x_{13}\,y_{12}\,x_{11}\,y_{10}\,x_{9}; \\ 
&R_4=\,y_{0}\,x_{1}\,y_{2}\,x_{2}\,x_{3}\,y_{1}\,x_{0}; &\\ 
&&\\
&X_5=\,y_{4}\,x_{6}\,x_{4}\,y_{6}\,y_{4}; &S_5=\,y_{9}\,y_{10}\,x_{11}\,y_{13}; \\ 
&Q_5=\,x_{10}\,y_{8}\,x_{8}\,x_{9}\,y_{7}\,y_{5}\,x_{5}\,x_{7}\,y_{9}; &T_5=\,x_{14}\,y_{12}\,x_{12}\,x_{13}\,y_{11}\,x_{10}; \\ 
&R_5=\,y_{0}\,y_{1}\,x_{3}\,x_{2}\,y_{3}\,y_{2}\,x_{1};&\\ 
\end{array}$
\end{center}

\begin{center}
$\begin{array}{p{0cm} l l}
&X_6=\,y_{4}\,x_{5}\,x_{3}\,x_{4}\,y_{4}; &S_6=\,x_{9}\,x_{10}\,x_{12}\,x_{11}\,x_{13}; \\ 
&Q_6=\,y_{9}\,y_{8}\,y_{6}\,y_{5}\,x_{6}\,y_{7}\,x_{8}\,x_{7}\,x_{9}; &T_6=\,y_{13}\,y_{12}\,y_{10}\,y_{11}\,y_{9}; \\ 
&R_6=\,x_{0}\,x_{1}\,y_{3}\,y_{1}\,x_{2}\,y_{2}\,y_{0}; &\\ 
&&\\
&X_7=\,y_{4}\,x_{4}\,x_{3}\,y_{5}\,y_{4};&S_7=\,y_{9}\,y_{11}\,y_{10}\,y_{12}\,y_{13}; \\ 
&Q_7=\,x_{9}\,x_{7}\,x_{8}\,y_{6}\,x_{5}\,y_{7}\,x_{6}\,y_{8}\,y_{9}; &T_7=\,x_{13}\,x_{11}\,x_{12}\,x_{10}\,x_{9}; \\ 
&R_7=\,y_{0}\,y_{2}\,y_{1}\,y_{3}\,x_{2}\,x_{1}\,x_{0};&\\ 
&&\\
&X_8=\,y_{4}\,y_{5}\,y_{3}\,x_{3}\,y_{4}; &S_8=\,x_{9}\,y_{11}\,y_{12}\,x_{13}; \\ 
&Q_8=\,y_{9}\,x_{8}\,y_{7}\,x_{7}\,x_{6}\,x_{5}\,y_{6}\,y_{8}\,x_{9}; &T_8=\,y_{13}\,x_{12}\,y_{10}\,x_{10}\,x_{11}\,y_{9}.\\ 
&R_8=\,x_{0}\,y_{2}\,x_{4}\,x_{2}\,y_{1}\,x_{1}\,y_{0};&\\ 
\end{array}$
\end{center}

\noindent {\bf Case 3}: $t_1=6$ and $q=16$. See Figures \ref{fig:case3.1}--\ref{fig:case3.9} in Appendix \ref{A:B1}. 

 \medskip
 \begin{center}
$\begin{array}{l l}
X_0=y_5\,x_{6}\,y_{8}\,x_{7}\,x_{8}\,y_{7}\,y_{5};&S_0=\,y_{11}\,x_{13}\,y_{13}\,y_{12}\,x_{14}\,y_{15};  \\ 
Q_0=\,x_{11}\,x_{9}\,y_{9}\,y_{10}\,x_{10}\,y_{11}; &T_0=\,x_{15}\,y_{14}\,x_{12}\,x_{11}; \\
R_0=\,y_{0}\,x_{2}\,y_{2}\,y_{1}\,x_{3}\,x_{4}\,y_{6}\,y_{4}\,x_{5}\,y_{3}\,x_{1}\,x_{0}; & \\ 
&\\
X_1=\,x_{5}\,y_{5}\,x_{7}\,y_{6}\,y_{8}\,y_{7}\,x_{5};  &S_1=\,y_{12}\,x_{12}\,x_{14}\,y_{14}\,y_{16}; \\
Q_1=\,y_{11}\,y_{9}\,x_{9}\,x_{8}\,x_{10}\,y_{10}\,y_{12};  &T_1=\,y_{15}\,x_{13}\,x_{11}\,y_{13}\,y_{11}; \\ 
R_1=\,y_{1}\,x_{1}\,x_{3}\,y_{3}\,x_{4}\,x_{6}\,y_{4}\,x_{2}\,x_{0}\,y_{2}\,y_{0}; & \\ 
&\\
X_2=\,y_{5}\,x_{5}\,y_{7}\,x_{8}\,x_{7}\,x_{6}\,y_{5}; &S_2= \,y_{11}\,y_{13}\,x_{11}\,x_{13}\,y_{15}; \\ 
Q_2=\,y_{12}\,y_{10}\,x_{9}\,x_{10}\,y_{8}\,y_{9}\,y_{11}; &T_2= \,y_{16}\,y_{14}\,x_{14}\,x_{12}\,y_{12}; \\
R_2=\,y_{0}\,y_{2}\,x_{0}\,x_{2}\,y_{4}\,y_{6}\,x_{4}\,y_{3}\,x_{3}\,x_{1}\,y_{1}; &\\ 
&\\
X_3=\,x_{5}\,x_{6}\,x_{8}\,y_{6}\,y_{7}\,x_{7}\,x_{5}; &S_3=\,x_{12}\,y_{13}\,x_{15}\,y_{15}\,y_{14}\,x_{16}; \\ 
Q_3=\,y_{12}\,x_{10}\,y_{9}\,y_{8}\,x_{9}\,x_{11}\,y_{11}\,y_{10}\,x_{12}; &T_3=\,y_{16}\,x_{14}\,x_{13}\,y_{12}; \\
R_3=\,x_{1}\,y_{2}\,x_{4}\,y_{4}\,y_{3}\,y_{5}\,x_{3}\,x_{2}\,y_{1}; &\\ 
\end{array}$
\end{center}

\begin{center}
$\begin{array}{l l}
X_4=\,x_{6}\,x_{7}\,y_{9}\,y_{7}\,y_{6}\,x_{8}\,x_{6}; &S_4=\,x_{11}\,y_{12}\,x_{13}\,y_{14}\,x_{15}; \\ 
Q_4=\,y_{11}\,x_{10}\,x_{9}\,y_{8}\,y_{10}\,x_{11}; &T_4=\,y_{15}\,x_{14}\,y_{13}\,x_{12}\,y_{11}; \\ 
R_4=\,x_{0}\,y_{1}\,x_{2}\,y_{3}\,y_{4}\,y_{5}\,x_{4}\,x_{5}\,x_{3}\,y_{2}\,x_{1}\,y_{0}; &\\ 
&\\
X_5=\,y_{6}\,x_{5}\,x_{7}\,x_{9}\,y_{7}\,x_{6}; & S_5=\,y_{12}\,y_{13}\,x_{14}\,y_{16}; \\ 
Q_5=\,x_{12}\,y_{10}\,y_{9}\,x_{8}\,y_{8}\,x_{10}\,y_{12}; & T_5=\,x_{16}\,y_{14}\,x_{13}\,y_{11}\,x_{11}\,x_{12}; \\ 
R_5=\, y_1\,y_{2}\,x_{3}\,y_{4}\,x_{4}\,y_{5}\,y_{3}\,x_{2}\,y_{0}\,x_{0}\, x_1; &\\ 
&\\
X_6=\,y_{6}\,x_{7}\,y_{7}\,y_{8}\,x_{6}\,x_{5}\,y_{6}; &S_6=\,x_{12}\,x_{13}\,x_{15}\,x_{14}\,x_{16}; \\ 
Q_6=\,y_{11}\,x_{9}\,y_{10}\,x_{8}\,y_{9}\,x_{11}\,x_{10}\,x_{12}; &T_6=\,y_{15}\,y_{13}\,y_{14}\,y_{12}\,y_{11}; \\ 
R_6=\,x_{1}\,x_{2}\,x_{4}\,x_{3}\,y_{5}\,y_{4}\,y_{2}\,y_{3}\,y_{1}\,y_{0}; &\\ 
&\\
X_7=\,y_{5}\,y_{6}\,x_{6}\,y_{7}\,y_{9}\,x_{7}\,y_{5}; &S_7=\,y_{11}\,y_{12}\,y_{14}\,y_{13}\,y_{15}; \\ 
Q_7=\,x_{12}\,x_{10}\,x_{11}\,y_{10}\,y_{8}\,x_{8}\,x_{9}\,y_{11}; &T_7=\,x_{16}\,x_{14}\,x_{15}\,x_{13}\,x_{12}; \\ 
R_7=\,y_{0}\,y_{1}\,y_{3}\,y_{2}\,y_{4}\,x_{3}\,x_{5}\,x_{4}\,x_{2}\,x_{1}; &\\ 
&\\
X_8=\,y_{5}\,y_{7}\,x_{9}\,x_{7}\,y_{8}\,y_{6}\,y_{5}; & S_8=\,y_{11}\,x_{12}\,y_{14}\,y_{15}; \\
Q_8=\,x_{11}\,y_{9}\,x_{10}\,x_{8}\,y_{10}\,y_{11}; &T_8=\,x_{15}\,y_{13}\,x_{13}\,x_{14}\,y_{12}\,x_{11}.   \\ 
R_8=\,y_{0}\,x_{1}\,y_{3}\,x_{5}\,y_{4}\,x_{6}\,x_{4}\,y_{2}\,x_{2}\,x_{3}\,y_{1}\,x_{0};\\ 
\end{array}$
\end{center}

\noindent {\bf Case 4}:  $t_1=6$ and $q=20$. See Figures \ref{fig:case4.1}--\ref{fig:case4.9} in Appendix \ref{A:B1}. 
 
 \medskip
 \begin{center}
$\begin{array}{l l}
X_0=y_3\,y_{5}\,x_{7}\,x_{5}\,y_{6}\,y_{4}\,y_{3}; &S_0=\,y_{13}\,x_{14}\,y_{16}\,y_{17}; \\ 
Q_0=\,x_{13}\,y_{11}\,x_{11}\,x_{12}\,y_{10}\,x_{9}\,y_{7}\,x_{8}\,x_{6}\,y_{8}\,y_{9}\,x_{10}\,y_{12}\,y_{13}; &T_0=\,x_{17}\,y_{15}\,x_{15}\,x_{16}\,y_{14}\,x_{13}; \\  
R_0=\,y_{0}\,x_{1}\,x_{3}\,y_{2}\,x_{4}\,x_{2}\,y_{1}\,x_{0}; \\ 
&\\
X_1=\,y_{3}\,y_{4}\,x_{4}\,y_{5}\,y_{7}\,x_{5}\,y_{3}; &S_1=\,y_{13}\,y_{14}\,y_{16}\,y_{15}\,y_{17}; \\ 
Q_1=\,x_{14}\,x_{12}\,x_{13}\,x_{11}\,x_{10}\,x_{8}\,x_{9}\,y_{8}\,y_{6}\,x_{6}\,x_{7}\,y_{9}\,y_{10}\,y_{12}\,y_{11}\,\,y_{13}; &T_1=\,x_{18}\,x_{16}\,x_{17}\,x_{15}\,x_{14}; \\  
R_1=\,y_{0}\,y_{1}\,x_{3}\,x_{2}\,y_{2}\,x_{1}; \\ 
&\\
X_2=\,y_{4}\,x_{5}\,y_{5}\,y_{6}\,x_{4}\,x_{3}\,y_{4}; &S_2=\,x_{14}\,x_{15}\,x_{17}\,x_{16}\,x_{18}; \\ 
Q_2=\,y_{13}\,y_{11}\,y_{12}\,y_{10}\,y_{9}\,x_{7}\,y_{8}\,x_{6}\,y_{7}\,x_{9}\,x_{8}\,x_{10}\,x_{11}\,x_{13}\,x_{12}\,x_{14}; &T_2=\,y_{17}\,y_{15}\,y_{16}\,y_{14}\,y_{13}; \\ 
R_2=\,x_{1}\,y_{2}\,x_{2}\,y_{3}\,y_{1}\,y_{0}; \\ 
&\\
X_3=\,y_{3}\,x_{5}\,y_{4}\,y_{6}\,y_{5}\,x_{3}y_3; \\ 
Q_3=\,x_{14}\,y_{12}\,x_{11}\,y_{9}\,x_{9}\,x_{10}\,y_{8}\,y_{7}\,x_{7}\,x_{6}\,x_{8}\,y_{10}\,y_{11}\,x_{12}\,y_{14}; &S_3=\,y_{14}\,y_{15}\,x_{16}\,y_{18}; \\ 
R_3=\,y_{1}\,x_{2}\,x_{4}\,y_{2}\,y_{0}\,x_{0}\,x_{1}; &T_3=\,x_{18}\,y_{16}\,x_{15}\,y_{13}\,x_{13}\,x_{14}; \\ 
&\\
X_4=\,y_{3}\,x_{4}\,y_{6}\,x_{5}\,x_{6}\,y_{5}\,y_{3}; &S_4=\,x_{13}\,y_{14}\,x_{15}\,y_{16}\,x_{17}; \\ 
Q_4=\,y_{13}\,x_{12}\,y_{11}\,x_{10}\,y_{9}\,x_{8}\,x_{7}\,y_{7}\,y_{8}\,x_{9}\,y_{10}\,x_{11}\,y_{12}\,x_{13}; &T_4=\,y_{17}\,x_{16}\,y_{15}\,x_{14}\,y_{13}; \\ 
R_4=\,x_{0}\,y_{1}\,y_{2}\,y_{4}\,x_{2}\,x_{3}\,x_{1}\,y_{0}; \\ 
\end{array}$
\end{center}

\begin{center}
$\begin{array}{l l}
X_5=\,y_{3}\,x_{3}\,y_{5}\,x_{6}\,x_{5}\,x_{4}y_3; &S_5=\,x_{14}\,y_{15}\,x_{17}\,y_{17}\,y_{16}\,x_{18}; \\ 
Q_5=\,y_{14}\,x_{12}\,x_{11}\,y_{10}\,x_{8}\,y_{7}\,y_{6}\,x_{7}\,x_{9}\,y_{9}\,y_{8}\,x_{10}\,y_{11}\,x_{13}\,y_{13}\,y_{12}\,x_{14}; &T_5=\,y_{18}\,x_{16}\,x_{15}\,y_{14}; \\ 
R_5=\,x_{1}\,x_{2}\,y_{4}\,y_{2}\,y_{1}; \\ 
&\\
X_6=\,x_{3}\,x_{4}\,x_{6}\,y_{4}\,y_{5}\,x_{5}\,x_{3}; &S_6=\,y_{13}\,y_{15}\,x_{13}\,x_{15}\,y_{17}; \\ 
Q_6=\,y_{14}\,y_{12}\,x_{12}\,x_{10}\,y_{10}\,y_{8}\,x_{7}\,x_{8}\,y_{6}\,y_{7}\,y_{9}\,y_{11}\,x_{9}\,x_{11}\,y_{13}; &T_6=\,y_{18}\,y_{16}\,x_{16}\,x_{14}\,y_{14}; \\ 
R_6=\,y_{0}\,x_{2}\,x_{0}\,y_{2}\,y_{3}\,x_{1}\,y_{1}; \\ 
&\\
X_7=\,x_{3}\,x_{5}\,x_{7}\,y_{5}\,x_{4}\,y_{4}\,x_{3}; &S_7=\,y_{14}\,x_{14}\,x_{16}\,y_{16}\,y_{18}; \\ 
Q_7=\,y_{13}\,x_{11}\,x_{9}\,y_{11}\,y_{9}\,y_{7}\,x_{6}\,y_{6}\,x_{8}\,y_{8}\,y_{10}\,x_{10}\,x_{12}\,y_{12}\,y_{14}; &T_7=\,y_{17}\,x_{15}\,x_{13}\,y_{15}\,y_{13}; \\ 
R_7=\,y_{1}\,x_{1}\,y_{3}\,y_{2}\,x_{0}\,x_{2}\,y_{0}; \\ 
&\\
X_8=\,x_{4}\,x_{5}\,y_{7}\,y_{5}\,y_{4}\,x_{6}\,x_{4}; &S_8=\,y_{13}\,x_{15}\,y_{15}\,y_{14}\,x_{16}\,y_{17}; \\ 
Q_8=\,x_{13}\,y_{12}\,x_{10}\,x_{9}\,x_{7}\,y_{6}\,y_{8}\,x_{8}\,y_{9}\,x_{11}\,y_{11}\,y_{10}\,x_{12}\,y_{13}; &T_8=\,x_{17}\,y_{16}\,x_{14}\,x_{13}. \\ 
R_8=\,y_{0}\,y_{2}\,x_{3}\,y_{1}\,y_{3}\,x_{2}\,x_{1}\,x_{0}; 
\end{array}$
 \end{center}

 \subsection{The case $t_1+q \equiv 0\ (\textrm{mod}\ 4)$}
  \label{sub:case2}
  \noindent {\bf Case 1}: $t_1=4$ and $q=16$. See Figures \ref{fig:case5.1}--\ref{fig:case5.7} in Appendix \ref{A:B2}. 
 \begin{center}
$\begin{array}{p{0cm} l l}
&X_0=x_7\,y_{8}\,y_{7}\,x_{8}\,x_{7}; &S_0=\,y_{10}\,x_{13}\,y_{13}\,x_{14}\,y_{14};\\
&Q_0=\,x_{11}\,x_{10}\,y_{9}\,x_{9}\,y_{10};&T_0=\,x_{15}\,x_{12}\,y_{12}\,y_{11}\,x_{11}; \\ 
&R_0=\,y_{0}\,x_{3}\,x_{4}\,y_{4}\,y_{3}\,y_{2}\,x_{5}\,y_{6}\,x_{6}\,y_{5}\,x_{2}\,y_{1}\,x_{1};& \\ 
&&\\
&X_1=\,x_{6}\,x_{9}\,y_{6}\,y_{9}\,x_{6}; &S_1=\,y_{11}\,x_{12}\,x_{13}\,y_{12}\,y_{15}; \\ 
&Q_1=\,x_{10}\,y_{10}\,x_{11}\,x_{8}\,y_{8}\,y_{11}; &T_1=\,x_{14}\,x_{15}\,y_{14}\,y_{13}\,x_{10}; \\ 
&R_1=\,y_{1}\,x_{2}\,x_{5}\,y_{5}\,x_{4}\,y_{7}\,x_{7}\,y_{4}\,x_{3}\,y_{2}\,y_{3}\,x_{0}; &\\ 
&&\\
\end{array}$
\end{center}

\noindent {\bf Case 2}: $t_1=4$ and $q=20$. See Figures \ref{fig:case6.1}--\ref{fig:case6.7} in Appendix \ref{A:B2}.
  \begin{center}
$\begin{array}{p{0cm} l l}
&X_0=x_{10}\,x_{11}\,y_{10}\,y_{11}\,x_{10}; &S_0=\,y_{12}\,y_{15}\,x_{15}\,x_{14}\,x_{17}\,x_{16}\,y_{16}; \\ 
&Q_0=\,y_{13}\,x_{12}\,x_{13}\,y_{12}; &T_0=\,y_{17}\,y_{14}\,y_{13}; \\ 
&R_0=\,y_{0}\,y_{3}\,y_{4}\,x_{3}\,x_{4}\,y_{7}\,y_{8}\,x_{8}\,x_{7}\,x_{6}\,x_{9}\,y_{9}\,y_{6}\,x_{5}\,y_{5}\,x_{2}\,y_{2}\,y_{1}; \\ 
&&\\
&X_1=\,x_{8}\,x_{11}\,y_{8}\,y_{11}\,x_{8}; &S_1=\,y_{13}\,x_{14}\,y_{14}\,y_{17}; \\ 
&Q_1=\,x_{12}\,y_{12}\,x_{13}\,x_{10}\,y_{10}\,y_{13}; &T_1=\,x_{16}\,x_{17}\,y_{16}\,x_{15}\,y_{15}\,x_{12}; \\ 
&R_1=\,y_{1}\,y_{2}\,y_{5}\,x_{6}\,y_{9}\,x_{9}\,y_{6}\,y_{7}\,x_{7}\,y_{4}\,x_{5}\,x_{4}\,x_{3}\,x_{2}\,y_{3}\,x_{0}; \\ 
&\\
&X_2=\,x_{8}\,y_{9}\,x_{10}\,x_{9}\,x_{8}; &S_2=\,x_{13}\,x_{14}\,y_{15}\,y_{14}\,x_{17}; \\ 
&Q_2=\,y_{12}\,y_{11}\,x_{11}\,x_{12}\,y_{13}\,y_{10}\,x_{13}; &T_2=\,y_{16}\,y_{17}\,x_{16}\,x_{15}\,y_{12}; \\
&R_2=\,x_{1}\,y_{2}\,x_{2}\,y_{5}\,x_{4}\,x_{5}\,y_{6}\,x_{6}\,x_{7}\,y_{8}\,y_{7}\,y_{4}\,y_{3}\,x_{3}\,y_{0}; &\\ 
\end{array}$
\end{center}

\begin{center}
$\begin{array}{p{0cm} l l}
&X_3=\,y_{8}\,x_{9}\,y_{10}\,y_{9}\,y_{8};&S_3=\,y_{13}\,y_{14}\,x_{15}\,y_{16}\,y_{15}\,x_{16}\,y_{17}; \\ 
&Q_3=\,x_{13}\,x_{12}\,x_{11}\,x_{10}\,y_{11}\,y_{12}\,y_{13}; &T_3=\,x_{17}\,x_{14}\,x_{13}; \\ 
&R_3=\,y_{1}\,x_{2}\,x_{3}\,y_{3}\,x_{4}\,y_{4}\,y_{5}\,y_{6}\,x_{7}\,x_{8}\,y_{7}\,x_{6}\,x_{5}\,y_{2}\,x_{1};& \\ 
&&\\
&X_4=\,x_{8}\,x_{9}\,y_{8}\,x_{11}\,x_{8};&S_4=\,y_{13}\,x_{13}\,y_{14}\,x_{14}\,y_{17}; \\
&Q_4=\,x_{12}\,y_{11}\,y_{10}\,x_{10}\,y_{13};&T_4=\,x_{16}\,y_{15}\,y_{12}\,x_{15}\,x_{12}; \\ 
&R_4=\,y_{1}\,x_{1}\,x_{2}\,x_{5}\,y_{4}\,x_{7}\,y_{6}\,y_{9}\,x_{6}\,y_{7}\,x_{4}\,y_{5}\,y_{2}\,y_{3}\,y_{0}\,x_{3}\,x_{0};& \\ 
&&\\
&X_5=\,x_{8}\,y_{11}\,y_{8}\,y_{9}\,x_{8}; &S_5=\,x_{12}\,x_{15}\,x_{16}; \\ 
&Q_5=\,y_{13}\,x_{10}\,x_{13}\,y_{10}\,x_{11}\,y_{12}\,x_{12}; &T_5=\,y_{17}\,y_{16}\,x_{17}\,y_{14}\,y_{15}\,x_{14}\,y_{13}; \\ 
&R_5=\,x_{0}\,y_{3}\,y_{2}\,x_{3}\,y_{4}\,x_{4}\,x_{7}\,y_{7}\,y_{6}\,x_{9}\,x_{6}\,y_{5}\,x_{5}\,x_{2}\,y_{1};& \\ 
&&\\
&X_6=\,x_{9}\,x_{10}\,y_{9}\,y_{10}\,x_{9};&S_6=\,x_{12}\,y_{15}\,y_{16}\,x_{16}; \\ 
&Q_6=\,y_{13}\,y_{12}\,x_{11}\,y_{11}\,x_{12};&T_6=\,y_{17}\,x_{14}\,x_{15}\,y_{14}\,x_{13}\,y_{13}. \\
&R_6=\,x_{0}\,x_{3}\,y_{2}\,x_{5}\,x_{6}\,y_{6}\,y_{5}\,y_{4}\,y_{7}\,x_{8}\,y_{8}\,x_{7}\,x_{4}\,y_{3}\,x_{2}\,x_{1}\,y_{1}; &\\ 
  \end{array}$
  \end{center}
 
 \noindent {\bf Case 3}: $t_1=6$ and $q=14$. See Figures \ref{fig:case7.1}--\ref{fig:case7.7} in Appendix \ref{A:B2}.
 \medskip
 \begin{center}
$\begin{array}{p{0cm} l l}
&X_0=x_4\,y_{7}\,x_{6}\,y_{5}\,y_{6}\,x_{7}\,x_{4}; &S_0=\,y_{11}\,x_{10}\,x_{11}\,y_{12}\,x_{12}\,y_{15}; \\
&Q_0=\,y_{10}\,y_{9}\,x_{9}\,y_{8}\,x_{8}\,y_{11}; &T_0=\,y_{14}\,y_{13}\,x_{13}\,y_{10}; \\
&R_0=\,y_{1}\,x_{0}\,x_{1}\,x_{2}\,x_{5}\,y_{2}\,x_{3}\,y_{4}\,y_{3}\,y_{0}; &\\ 
&&\\
&X_1=\,x_{4}\,x_{7}\,y_{7}\,y_{4}\,y_{5}\,x_{5}\,x_{4};&S_1=\,x_{11}\,y_{10}\,y_{13}\,y_{12}\,x_{15}; \\ 
&Q_1=\,y_{11}\,y_{8}\,y_{9}\,y_{6}\,x_{6}\,x_{9}\,x_{8}\,x_{11}; &T_1=\,y_{15}\,x_{12}\,x_{13}\,x_{10}\,y_{11}; \\
&R_1=\,x_{1}\,y_{0}\,y_{3}\,x_{0}\,x_{3}\,x_{2}\,y_{2}\,y_{1};& \\ 
&&\\
&X_2=\,y_{4}\,x_{5}\,y_{5}\,x_{6}\,y_{7}\,x_{7}\,y_{4}; &S_2=\,y_{11}\,x_{11}\,x_{12}\,y_{12}\,y_{15}; \\ 
&Q_2=\,x_{10}\,y_{10}\,x_{9}\,y_{6}\,y_{9}\,x_{8}\,y_{8}\,y_{11};&T_2=\,x_{14}\,y_{14}\,x_{13}\,y_{13}\,x_{10}; \\ 
&R_2=\,y_{1}\,x_{1}\,y_{2}\,x_{2}\,y_{3}\,x_{4}\,x_{3}\,x_{0};& \\ 
&&\\
&X_3=\,x_{4}\,x_{5}\,y_{4}\,x_{7}\,y_{8}\,y_{7}\,x_{4};&S_3=\,y_{10}\,x_{13}\,y_{14}; \\ 
&Q_3=\,y_{11}\,x_{8}\,y_{9}\,x_{6}\,y_{6}\,x_{9}\,y_{10};&T_3=\,y_{15}\,y_{12}\,x_{11}\,x_{10}\,y_{13}\,x_{12}\,y_{11}; \\ 
&R_3=\,y_{0}\,x_{1}\,x_{0}\,y_{3}\,x_{3}\,y_{2}\,y_{5}\,x_{2}\,y_{1};& \\ 
\end{array}$
\end{center}

\begin{center}
$\begin{array}{p{0cm} l l}
&X_4=\,x_{5}\,y_{6}\,y_{7}\,x_{8}\,x_{7}\,x_{6}\,x_{5};&S_4=\,x_{11}\,y_{11}\,x_{12}\,x_{15}; \\ 
&Q_4=\,y_{10}\,x_{10}\,x_{9}\,y_{9}\,y_{8}\,x_{11}; &T_4=\,y_{14}\,x_{14}\,x_{13}\,y_{12}\,y_{13}\,y_{10}; \\ 
&R_4=\,x_{1}\,y_{1}\,y_{2}\,y_{3}\,x_{2}\,y_{5}\,x_{4}\,y_{4}\,x_{3}\,y_{0}; &\\ 
&&\\
&X_5=\,x_{5}\,x_{6}\,x_{7}\,x_{8}\,y_{7}\,y_{6}\,x_{5}; &S_5=\,y_{10}\,y_{11}\,y_{12}\,x_{13}\,x_{14}\,y_{13}\,y_{14}; \\ 
&Q_5=\,x_{11}\,y_{8}\,x_{9}\,x_{10}\,y_{9}\,y_{10}; &T_5=\,x_{15}\,x_{12}\,x_{11}; \\ 
&R_5=\,y_{0}\,y_{1}\,x_{2}\,x_{3}\,y_{3}\,y_{4}\,x_{4}\,y_{5}\,y_{2}\,x_{1};& \\ 
&&\\
&X_6=\,y_{5}\,y_{4}\,y_{7}\,y_{8}\,x_{7}\,y_{6}\,y_{5}; &S_6=\,x_{10}\,x_{13}\,x_{12}\,y_{13}\,x_{14}; \\ 
&Q_6=\,x_{11}\,x_{8}\,x_{9}\,x_{6}\,y_{9}\,x_{10}; &T_6=\,x_{15}\,y_{12}\,y_{11}\,y_{10}\,x_{11}. \\ 
&R_6=\,x_{0}\,y_{1}\,y_{0}\,x_{3}\,x_{4}\,y_{3}\,y_{2}\,x_{5}\,x_{2}\,x_{1};& \\ 
\end{array}$
\end{center}

\noindent {\bf Case 4}: $t_1=6$ and $q=18$. See Figures \ref{fig:case8.1}--\ref{fig:case8.7} in Appendix \ref{A:B2}.
 
\begin{center}
$\begin{array}{p{0cm} l l}
&X_0=x_6\,y_{9}\,x_{8}\,y_{7}\,y_{8}\,x_{9}\,x_{6}; &S_0=\,x_{13}\,y_{12}\,y_{13}\,x_{14}\,y_{14}\,x_{17}; \\ 
&Q_0=\,x_{12}\,y_{11}\,x_{11}\,y_{10}\,x_{10}\,x_{13}; &T_0=\,x_{16}\,x_{15}\,y_{15}\,x_{12}; \\ 
&R_0=\,x_{1}\,y_{0}\,y_{1}\,x_{2}\,y_{5}\,y_{6}\,x_{5}\,x_{4}\,x_{7}\,y_{4}\,x_{3}\,y_{2}\,y_{3}\,x_{0}; &\\  
&&\\
&X_1=\,x_{7}\,x_{6}\,x_{9}\,y_{9}\,y_{6}\,y_{7}\,x_{7};&S_1=\,y_{13}\,x_{13}\,x_{14}\,y_{17}; \\
&Q_1=\,y_{12}\,x_{12}\,x_{11}\,x_{8}\,y_{8}\,y_{11}\,x_{10}\,y_{10}\,y_{13}; &T_1=\,y_{16}\,x_{16}\,y_{15}\,y_{14}\,x_{15}\,y_{12}; \\
&R_1=\,y_{1}\,x_{1}\,x_{2}\,y_{2}\,y_{5}\,x_{4}\,x_{5}\,y_{4}\,y_{3}\,x_{3}\,y_{0}; &\\ 
&&\\
\end{array}$
\end{center}

\pagebreak

 \subsection{Special cases}
 \label{sub:spec}

 A $(\vec{C}_4,\vec{C}_{12})$-factorization of $W^*_{16}$:
 
\begin{center}
$\begin{array}{l l}
&F_0=\{ \,x_{5}\,x_{4}\,y_{3}\,y_{4}\,x_{5},\,\, \,x_{0}\,x_{3}\,x_{2}\,y_{2}\,y_{5}\,y_{6}\,x_{1}\,x_{6}\,y_{1}\,y_{0}\,y_{7}\,x_{7}\,x_{0}\}; \\ 
&F_1=\{ \,x_{6}\,y_{6}\,x_{5}\,y_{5}\,x_{6},\,\, \,x_{0}\,y_{1}\,y_{2}\,x_{2}\,x_{1}\,y_{0}\,x_{7}\,y_{4}\,x_{3}\,y_{3}\,x_{4}\,y_{7}\,x_{0}\}; \\ 
&F_2=\{ \,x_{5}\,x_{6}\,y_{7}\,x_{4}\,x_{5},\,\, \,x_{0}\,x_{1}\,y_{6}\,y_{5}\,y_{2}\,y_{1}\,x_{2}\,x_{3}\,y_{4}\,x_{7}\,y_{0}\,y_{3}\,x_{0}\}; \\ 
&F_3=\{ \,y_{6}\,x_{6}\,x_{7}\,y_{7}\,y_{6},\,\, \,x_{0}\,y_{0}\,y_{1}\,x_{1}\,x_{2}\,y_{3}\,y_{2}\,x_{5}\,y_{4}\,y_{5}\,x_{4}\,x_{3}\,x_{0}\}; \\ 
&F_4=\{ \,x_{2}\,y_{5}\,y_{4}\,y_{3}\,x_{2},\,\, \,x_{0}\,y_{7}\,y_{0}\,x_{1}\,y_{2}\,x_{3}\,x_{4}\,x_{7}\,x_{6}\,x_{5}\,y_{6}\,y_{1}\,x_{0}\}; \\ 
&F_5=\{ \,y_{3}\,y_{0}\,x_{3}\,y_{2}\,y_{3},\,\, \,x_{0}\,x_{7}\,y_{6}\,y_{7}\,y_{4}\,x_{4}\,y_{5}\,x_{5}\,x_{2}\,y_{1}\,x_{6}\,x_{1}\,x_{0}\}; \\ 
&F_6=\{ \,x_{0}\,y_{3}\,x_{3}\,y_{0}\,x_{0},\,\, \,x_{1}\,y_{1}\,y_{6}\,x_{7}\,x_{4}\,y_{4}\,y_{7}\,x_{6}\,y_{5}\,x_{2}\,x_{5}\,y_{2}\,x_{1}\}. \\ 
\end{array}$
\end{center}
 
\noindent A $(\vec{C}_6,\vec{C}_8)$-factors of $W^*_{14}$:

\begin{center}
$\begin{array}{l l}
&F_0=\{  x_{0}\,y_{1}\,y_{3}\,x_{5}\,y_{6}\,y_{0}\,x_{0} ,\,\,  x_{1}\,x_{2}\,y_{2}\,x_{3}\,y_{4}\,x_{4}\,y_{5}\,x_{6}\,x_{1} \}; \\ 
&F_1=\{  y_{0}\,y_{1}\,x_{3}\,y_{3}\,x_{4}\,x_{5}\,y_{0} ,\,\,  x_{0}\,x_{1}\,y_{2}\,x_{2}\,y_{4}\,y_{5}\,y_{6}\,x_{6}\,x_{0} \}; \\ 
&F_2=\{  x_{5}\,y_{5}\,y_{0}\,x_{1}\,y_{3}\,x_{3}\,x_{5} ,\,\,  x_{0}\,x_{2}\,x_{4}\,x_{6}\,y_{1}\,y_{2}\,y_{4}\,y_{6}\,x_{0} \}; \\ 
&F_3=\{  y_{2}\,y_{3}\,y_{5}\,x_{5}\,x_{0}\,y_{0}\,y_{2} ,\,\,  x_{1}\,y_{1}\,x_{2}\,x_{3}\,x_{4}\,y_{4}\,x_{6}\,y_{6}\,x_{1} \}; \\ 
&F_4=\{  y_{3}\,y_{4}\,x_{5}\,x_{6}\,y_{0}\,x_{2}\,y_{3} ,\,\,  x_{0}\,y_{2}\,x_{4}\,y_{6}\,y_{1}\,x_{1}\,x_{3}\,y_{5}\,x_{0} \}; \\ 
&F_5=\{  y_{6}\,y_{5}\,y_{4}\,x_{3}\,y_{2}\,y_{0}\,y_{6} ,\,\,  x_{0}\,x_{5}\,y_{3}\,y_{1}\,x_{6}\,x_{4}\,x_{2}\,x_{1}\,x_{0} \}; \\ 
&F_6=\{  x_{2}\,x_{0}\,x_{6}\,x_{5}\,y_{4}\,y_{3}\,x_{2} ,\,\,  y_{0}\,y_{5}\,x_{3}\,x_{1}\,y_{6}\,x_{4}\,y_{2}\,y_{1}\,y_{0} \}; \\ 
&F_7=\{  y_{6}\,x_{5}\,x_{4}\,x_{3}\,x_{2}\,y_{1}\,y_{6} ,\,\,  x_{0}\,y_{5}\,y_{3}\,x_{1}\,y_{0}\,x_{6}\,y_{4}\,y_{2}\,x_{0} \}; \\ 
&F_8=\{  x_{6}\,y_{5}\,x_{4}\,y_{3}\,y_{2}\,x_{1}\,x_{6} ,\,\,  x_{0}\,y_{6}\,y_{4}\,x_{2}\,y_{0}\,x_{5}\,x_{3}\,y_{1}\,x_{0} \}. \\
\end{array}$
\end{center}

\noindent A $(\vec{C}_6,\vec{C}_{10})$-factorization of $W^*_{16}$:
  
\begin{center}
$\begin{array}{l l}
&F_0=\{ \,y_{1}\,x_{0}\,x_{1}\,y_{6}\,y_{5}\,y_{2}\,y_{1},\,\, \,y_{0}\,y_{7}\,x_{4}\,y_{4}\,y_{3}\,x_{3}\,x_{2}\,x_{5}\,x_{6}\,x_{7}\,y_{0}\}; \\ 
&F_1=\{ \,x_{4}\,x_{7}\,x_{0}\,y_{1}\,y_{0}\,y_{3}\,x_{4},\,\, \,x_{1}\,x_{2}\,x_{3}\,y_{4}\,y_{7}\,y_{6}\,x_{5}\,y_{2}\,y_{5}\,x_{6}\,x_{1}\}; \\ 
&F_2=\{ \,y_{1}\,x_{2}\,y_{2}\,x_{5}\,y_{6}\,x_{6}\,y_{1},\,\, \,x_{0}\,y_{7}\,x_{7}\,x_{4}\,y_{5}\,y_{4}\,x_{3}\,y_{3}\,y_{0}\,x_{1}\,x_{0}\}; \\ 
&F_3=\{ \,x_{0}\,y_{3}\,y_{2}\,x_{1}\,y_{0}\,x_{3}\,x_{0},\,\, \,y_{1}\,y_{6}\,x_{7}\,y_{4}\,x_{5}\,x_{4}\,y_{7}\,x_{6}\,y_{5}\,x_{2}\,y_{1}\}; \\ 
&F_4=\{ \,y_{0}\,y_{1}\,x_{6}\,y_{7}\,x_{0}\,x_{3}\,y_{0},\,\, \,x_{1}\,y_{2}\,y_{3}\,x_{2}\,y_{5}\,x_{4}\,x_{5}\,y_{4}\,x_{7}\,y_{6}\,x_{1}\}; \\ 
&F_5=\{ \,x_{1}\,x_{6}\,x_{5}\,y_{5}\,y_{6}\,y_{1}\,x_{1},\,\, \,x_{0}\,y_{0}\,x_{7}\,y_{7}\,y_{4}\,x_{4}\,x_{3}\,y_{2}\,x_{2}\,y_{3}\,x_{0}\}; \\ 
&F_6=\{ \,y_{7}\,y_{0}\,x_{0}\,x_{7}\,x_{6}\,y_{6}\,y_{7},\,\, \,x_{1}\,y_{1}\,y_{2}\,x_{3}\,x_{4}\,y_{3}\,y_{4}\,y_{5}\,x_{5}\,x_{2}\,x_{1}\}.
\end{array}$
\end{center}

\noindent A $(\vec{C}_6,\vec{C}_{12})$-factorization of $W^*_{18}$:
 
\begin{center}
$\begin{array}{l lll}
&F_0=\{y_{1}\,y_{8}\,x_{8}\,x_{6}\,x_{5}\,y_{3}\,y_{1},\,\, x_{0}\,x_{1}\,y_{0}\,x_{7}\,y_{5}\,x_{3}\,y_{4}\,x_{2}\,y_{2}\,x_{4}\,y_{6}\,y_{7}\,x_{0}\}; \\ 
&F_1=\{x_{2}\,y_{0}\,y_{2}\,x_{3}\,x_{5}\,x_{4}\,x_{2},\,\, x_{0}\,x_{7}\,y_{8}\,x_{1}\,y_{3}\,y_{4}\,y_{6}\,x_{6}\,y_{5}\,y_{7}\,x_{8}\,y_{1}\,x_{0}\}; \\ 
&F_2=\{y_{0}\,x_{0}\,y_{8}\,x_{6}\,x_{8}\,y_{7}\,y_{0},\,\, x_{1}\,x_{3}\,y_{3}\,x_{5}\,y_{4}\,y_{5}\,x_{7}\,y_{6}\,x_{4}\,y_{2}\,y_{1}\,x_{2}\,x_{1}\}; \\ 
&F_3=\{x_{5}\,y_{6}\,y_{4}\,x_{4}\,x_{6}\,x_{7}\,x_{5},\,\, x_{0}\,x_{8}\,x_{1}\,x_{2}\,y_{1}\,x_{3}\,y_{5}\,y_{3}\,y_{2}\,y_{0}\,y_{7}\,y_{8}\,x_{0}\}; \\ 
&F_4=\{x_{1}\,y_{2}\,y_{3}\,x_{4}\,x_{3}\,y_{1}\,x_{1},\,\, x_{0}\,y_{0}\,x_{2}\,y_{4}\,x_{6}\,y_{8}\,x_{7}\,y_{7}\,x_{5}\,y_{5}\,y_{6}\,x_{8}\,x_{0}\}; \\ 
&F_5=\{x_{2}\,y_{3}\,x_{3}\,x_{4}\,y_{4}\,y_{2}\,x_{2},\,\, x_{0}\,y_{1}\,x_{8}\,y_{8}\,y_{7}\,y_{5}\,x_{5}\,x_{6}\,y_{6}\,x_{7}\,y_{0}\,x_{1}\,x_{0}\}; \\ 
&F_6=\{x_{0}\,y_{2}\,y_{4}\,x_{5}\,x_{3}\,x_{2}\,x_{0},\,\, y_{0}\,x_{8}\,x_{7}\,x_{6}\,y_{7}\,y_{6}\,y_{5}\,x_{4}\,y_{3}\,x_{1}\,y_{8}\,y_{1}\,y_{0}\}; \\ 
&F_7=\{y_{0}\,y_{8}\,y_{6}\,x_{5}\,x_{7}\,x_{8}\,y_{0},\,\, x_{0}\,y_{7}\,x_{6}\,x_{4}\,y_{5}\,y_{4}\,y_{3}\,x_{2}\,x_{3}\,x_{1}\,y_{1}\,y_{2}\,x_{0}\}; \\ 
&F_8=\{x_{0}\,x_{2}\,x_{4}\,x_{5}\,y_{7}\,x_{7}\,x_{0},\,\, y_{0}\,y_{1}\,y_{3}\,y_{5}\,x_{6}\,y_{4}\,x_{3}\,y_{2}\,x_{1}\,x_{8}\,y_{6}\,y_{8}\,y_{0}\}. \\ 
\end{array}$
\end{center}

\pagebreak

\section{Illustration of all ($t_1, q$)-base tuples given in Appendix \ref{A:A}}

 \subsection{The case $t_1+q \equiv 2\ (\textrm{mod}\ 4)$}
 \label{A:B1}
 \noindent \uline{{\large Case 1: $t_1=4$ and $q=10$} \hfill}

\begin{figure} [H]
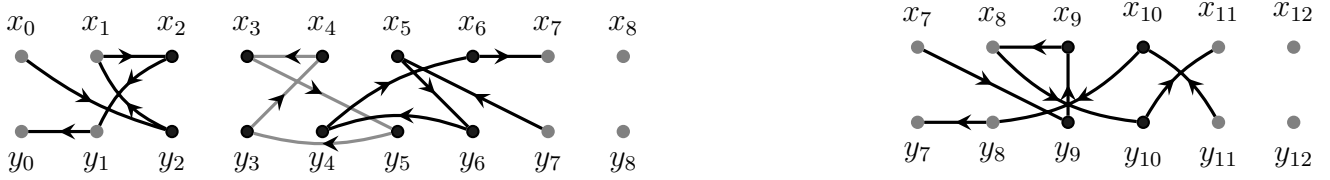

\begin{center}
%
\end{center}
\caption{The $(4, 10)$-base tuple $(X_0, Q_0, R_0, S_0, T_0)$.}
\label{fig:case1.1}
\end{figure}

\begin{figure} [H]
\begin{center}
%
\end{center}
\caption{The $(4, 10)$-base tuple $(X_1, Q_1, R_1, S_1, T_1)$.}
\label{fig:case1.2}
\end{figure}

\begin{figure} [H]
\begin{center}
%
\end{center}
\caption{The $(4, 10)$-base tuple $(X_2, Q_2, R_2, S_2, T_2)$.}
\label{fig:case1.3}
\end{figure}

\begin{figure} [H]
\begin{center}
%
\end{center}
\caption{The $(4, 10)$-base tuple $(X_8, Q_8, R_8, S_8, T_8)$.}
\label{fig:case1.9}
\end{figure}

\noindent \uline{{\large Case 2: $t_1=4$ and $q=14$} \hfill}

\begin{figure} [H]
\begin{center}
%
\end{center}
\caption{The $(4, 14)$-base tuple $(X_8, Q_8, R_8, S_8, T_8)$.}
\label{fig:case2.9}
\end{figure}

\pagebreak

\noindent \uline{{\large Case 3: $t_1=6$ and $q=16$} \hfill}

\begin{figure} [H]
\begin{center}
%
\end{center}
\caption{The $(6, 16)$-base tuple $(X_8, Q_8, R_8, S_8, T_8)$.}
\label{fig:case3.9}
\end{figure}

\pagebreak

\noindent \uline{{\large Case 4: $t_1=6$ and $q=20$} \hfill}

\begin{figure} [H]
\begin{center}
%
\end{center}
\caption{The $(6, 20)$-base tuple $(X_0, Q_0, R_0, S_0, T_0)$.}
\label{fig:case4.1}
\end{figure}

\begin{figure} [H]
\begin{center}
\hfill
%
\caption{The $(6, 20)$-base tuple $(X_5, Q_5, R_5, S_5, T_5)$.}
\end{center}
\label{fig:case4.6}
\end{figure}

\begin{figure} [H]
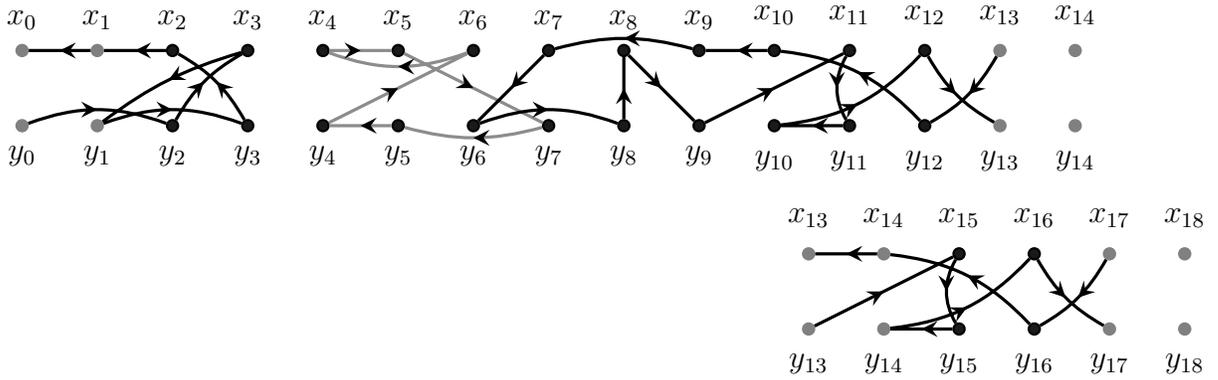

\begin{center}
%
\end{center}
\caption{The $(6, 20)$-base tuple $(X_8, Q_8, R_8, S_8, T_8)$.}
\label{fig:case4.9}
\end{figure}	


\pagebreak

 \subsection{The case $t_1+q \equiv 0\ (\textrm{mod}\ 4)$}
 \label{A:B2}

\noindent \uline{{\large Case 1: $t_1=4$ and $q=16$} \hfill}

\begin{figure} [H]
\begin{center}
%
\end{center}
\caption{The $(4, 16)$-base tuple $(X_6, Q_6, R_6, S_6, T_6)$.}
\label{fig:case5.7}
\end{figure}

\noindent \uline{{\large Case 2: $t_1=4$ and $q=20$} \hfill}

\begin{figure} [H]
\hspace{0.4cm}
%

\caption{The $(4, 20)$-base tuple $(X_0, Q_0, R_0, S_0, T_0)$.}
\label{fig:case6.1}
\end{figure}

\begin{figure} [H]
\hspace{0.4cm}
%
\caption{The $(4, 20)$-base tuple $(X_1, Q_1, R_1, S_1, T_1)$.}
\label{fig:case6.2}
\end{figure}

\begin{figure} [H]
\hspace{0.4cm}
%

\caption{The $(4, 20)$-base tuple $(X_2, Q_2, R_2, S_2, T_2)$.}
\label{fig:case6.3}
\end{figure}

\begin{figure} [H]
\hspace{0.4cm}
%

\caption{The $(4, 20)$-base tuple $(X_6, Q_6, R_6, S_6, T_6)$.}
\label{fig:case6.7}
\end{figure}


\noindent \uline{{\large Case 3: $t_1=6$ and $q=14$} \hfill}

\begin{figure} [H]
\begin{center}
%
\end{center}
\caption{The $(6, 14)$-base tuple $(X_6, Q_6, R_6, S_6, T_6)$.}
\label{fig:case7.7}
\end{figure}


\pagebreak
\noindent \uline{{\large Case 4: $t_1=6$ and $q=18$} \hfill}

\begin{figure} [H]
\hspace{0.4cm}
%
\caption{The $(6, 18)$-base tuple $(X_6, Q_6, R_6, S_6, T_6)$.}
\label{fig:case8.7}
\end{figure}

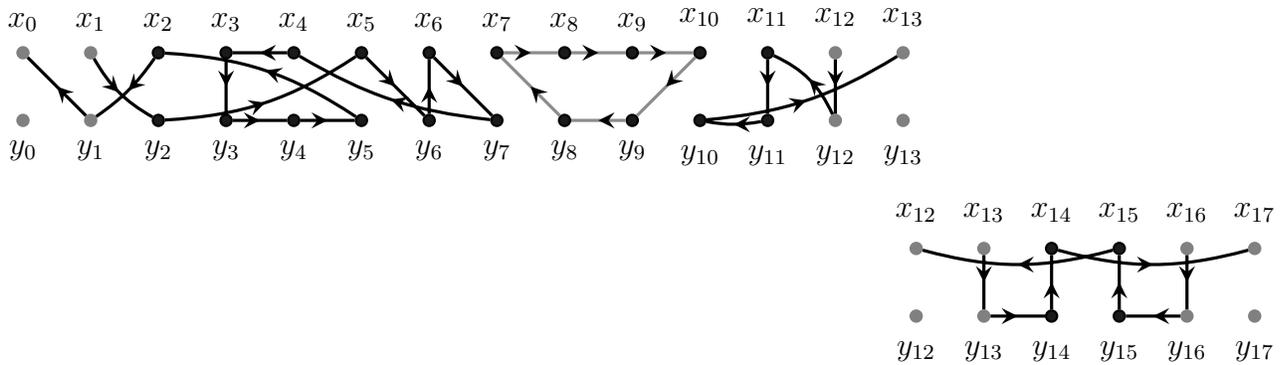
\end{document}